\documentclass[leqno,11pt]{amsart}%
\usepackage{eurosym}
\usepackage{amsmath}
\usepackage{mathrsfs}
\usepackage{amsfonts}
\usepackage{graphicx}
\usepackage{color}
\usepackage{amsfonts}
\usepackage{amssymb}%
\usepackage{esint}
\usepackage{todonotes}
\usepackage[abbrev]{amsrefs}
\usepackage{mathtools}
\usepackage{tikz-cd}
\usepackage{float}

\allowdisplaybreaks

\usepackage{amsfonts}
\usepackage{amsxtra}

\numberwithin{equation}{section}
  \providecommand{\Xint}[1]{\mathchoice
    {\XXint\displaystyle\textstyle{#1}}%
    {\XXint\textstyle\scriptstyle{#1}}%
    {\XXint\scriptstyle\scriptscriptstyle{#1}}%
    {\XXint\scriptscriptstyle\scriptscriptstyle{#1}}%
    \!\int}
  \providecommand{\XXint}[3]{{\setbox0=\hbox{$#1{#2#3}{\int}$}
      \vcenter{\hbox{$#2#3$}}\kern-.5\wd0}}
  
  \providecommand{\dashint}{\mathop{\Xint-}}

\newtheorem{theorem}{Theorem}[section]

\newtheorem{corollary}[theorem]{Corollary}

\newtheorem{example}[theorem]{Example}

\newtheorem{lemma}[theorem]{Lemma}

\newtheorem{remark}[theorem]{Remark}

\newtheorem{theoremalph}{Theorem}

\newtheorem{definitionalph}[theoremalph]{Definition}

\numberwithin{equation}{section}

\DeclareFontFamily{U}{mathx}{}
\DeclareFontShape{U}{mathx}{m}{n}{<-> mathx10}{}
\DeclareSymbolFont{mathx}{U}{mathx}{m}{n}
\DeclareMathAccent{\widehat}{0}{mathx}{"70}
\DeclareMathAccent{\widecheck}{0}{mathx}{"71}

\def\Mpl{\mathcal M_+}
\def\M{\mathcal M}

\def\om{\mathbb{R}^N}
\def\rn{\mathbb{R}^n}
\def\rM{\mathbb{R}^m}

\def\B{{\mathbb B}}

\def\M{{\mathcal M}}
\def\N{{\mathbb N}}

\newcommand{\smallD}{{\scriptscriptstyle D}}

\def\supp{{\rm supp\,}}

\newcommand{\R}{\mathbb{R}}

\newcommand{\barint}{
\rule[.036in]{.12in}{.009in}\kern-.16in \displaystyle\int }

\newcommand{\barcal}{\mbox{$ \rule[.036in]{.11in}{.007in}\kern-.128in\int $}}

\usepackage{enumitem}
\makeatletter
\let\@wraptoccontribs\wraptoccontribs
\makeatother

\mathchardef\mhyphen="2D

\title[
Sobolev inequalities for canceling  operators
 ]{
Sobolev inequalities for canceling  operators
}

\author {Dominic Breit, Andrea Cianchi \& Daniel Spector}

\address{Dominic Breit, 
  Institute of Mathematics\\
  TU Clausthal, 
Erzstra\ss e 1, 
Clausthal-Zellerfeld, 
38678 Germany} \email{dominic.breit@tu-clausthal.de}

\address{Andrea Cianchi, Dipartimento di Matematica e Informatica \lq\lq U. Dini"\\
Universit\`a di Firenze,
Viale Morgagni 67/a,
50134 Firenze,
Italy} \email{andrea.cianchi@unifi.it}

\address{Daniel Spector, Department of Mathematics \\ National Taiwan Normal University  No. 88, Section 4, Tingzhou Road, Wenshan District, Taipei City, Taiwan 116, R.O.C.
\newline
and
\newline
National Center for Theoretical Sciences\\
No. 1 Sec. 4 Roosevelt Rd., National Taiwan
University\\
Taipei, 106, Taiwan
\newline
and
\newline
Department of Mathematics \\ 
University of Pittsburgh, 
Pittsburgh, PA 
15261 
USA
}
\email{spectda@gapps.ntnu.edu.tw}

\subjclass[2020]{42B99, 46E35, 46E30}
\keywords{Sobolev inequalities; Riesz potentials;  canceling differential operators; co-canceling differential operators;    rearrangement-invariant spaces; Orlicz spaces}

\begin{document}

\begin{abstract} 
Sobolev type inequalities involving  homogeneous elliptic canceling differential operators and rearrangement-invariant norms on the Euclidean space are considered. They are characterized via considerably simpler one-dimensional Hardy type inequalities. As a consequence, they are shown to hold exactly for the same norms as their counterparts depending on the standard gradient operator of the same order. The results offered provide a unified framework for the theory of Sobolev embeddings for the elliptic canceling operators. They build upon and incorporate earlier fundamental contributions dealing with the endpoint case of $L^1$-norms. They also include previously available results  for the symmetric gradient, a prominent instance of an elliptic canceling operator.
 In particular, the optimal rearrangement-invariant target norm associated with any given domain norm in a Sobolev inequality 
 for any elliptic canceling operator is exhibited. Its explicit form is  
 detected  
  for specific families of rearrangement-invariant spaces, such as the Orlicz spaces and the Lorentz-Zygmund spaces. Especially relevant instances of inequalities for domain spaces neighboring   $L^1$ are singled out.
\end{abstract}

\maketitle


\section{Introduction}\label{intro}

The present work is aimed at providing a comprehensive approach to  embeddings of Sobolev type involving   \emph{rearrangement-invariant} norms   in $\rn$ of any sort and any \emph{homogeneous elliptic canceling} differential operator.
The Sobolev seminorm built upon a differential operator of this kind can be essentially weaker than that associated with the full gradient operator of the same order. A central consequence of our results 
 is that, this notwithstanding,  all Sobolev embeddings  for elliptic canceling operators 
 share their rearrangement-invariant domain and target spaces  with  Sobolev embeddings for  the canonical  gradient operator of the same order.

  The  classical form of the inequalities named  after Sobolev was established  in \cite{sobolev} and asserts that for $n \geq 2$ and $1< p<\frac nk$ one has
\begin{align}
    \label{may1}
    {\|u\|_{L^{\frac{np}{n-k p}}(\rn, \mathbb{R}^\ell)}\leq c \|\nabla ^k u\|_{L^p(\rn, \mathbb{R}^{\ell \times n^k})}}
\end{align}
for every $k$-times  weakly differentiable function $u: \rn \to \R^\ell$ decaying at infinity in a suitable sense (ruling out any polynomial of degree not exceeding $k-1$) and for some constant $c$ independent of $u$. Here, $n,\ell, k \in \N$, and $\nabla ^ku$ denotes the tensor of the $k$th order weak derivatives of $u$.  The inequality \eqref{may1} is also valid for $p=1$, though the proof given by Sobolev in \cite{sobolev} 
via a pointwise estimate for a weakly differentiable function in terms of a Riesz potential of its $k$th order weak derivatives and the boundedness properties of this potential is insufficient to deduce this endpoint case without further refinement.  The validity of the inequality for $p=1$ was first obtained by Gagliardo  \cite{gagliardo} and Nirenberg \cite{nirenberg} based on one-dimensional integration and H\"older's inequality and by Maz'ya \cite{mazya} and Federer and Fleming \cite{federer} via an isoperimetric inequality and the coarea formula.  This endpoint inequality is in some sense fundamental, as the remaining range $1<p<\frac nk$ of \eqref{may1} can be deduced from it.

Versions of the inequality \eqref{may1} with the operator $\nabla ^k$ replaced with different $k$-th order homogeneous differential operators, possibly skipping some of the partial derivatives, are available in the literature. A classical instance, for $k=1$ and $\ell =n$, is provided by the Sobolev inequality for the symmetric gradient operator.
The symmetric gradient is a distinguished member in the class of elliptic canceling operators, as defined in \cite[Definition 1.2]{VS3} and recalled in Section \ref{S:sobolev}, which render an inequality of the form \eqref{may1} true. Namely, one has that
\begin{align}
    \label{sob-cancelp}
 \|u\|_{L^{\frac{np}{n-kp}}(\rn, \mathbb{R}^\ell)}\leq c \|\mathcal A_k(D) u\|_{L^p(\rn, \R^m)}
\end{align}
for every $k$th order homogeneous elliptic canceling operator $\mathcal A_k(D)$, every  $u$ decaying near infinity, and some 
some constant $c$ independent of $u$. Here, the dimension $m$ depends on $\mathcal A_k(D)$.
\\
For $p>1$, this is a straightforward consequence of the bound
\begin{align}
    \label{korn}
   \|\nabla ^k u\|_{L^p(\rn, \R^{\ell\times n^k})} \leq c \|\mathcal A_k(D)u\|_{L^p(\rn, \R^m)}
\end{align}
  for every function $u$ decaying to zero near infinity and some constant $c$ independent of $u$, which goes back to \cite{CZ}. Remarkably, while the inequality \eqref{korn} fails for $p=1$ (and also for $p=\infty$), 
the inequality \eqref{sob-cancelp} continues to hold, though one must argue by another method.  This discovery of surprising Sobolev inequalities in $L^1$ originated in the pioneering work of Bourgain and Brezis \cite{BourgainBrezis2004, BourgainBrezis2007} before Van Schaftingen's introduction of the canceling condition in \cite{VS3}, which is necessary and sufficient for the validity of the inequality \eqref{sob-cancelp} when $p=1$. This area was subsequently developed in a number of contributions including \cites{  BourgainBrezisMironescu, BourgainBrezis2004, BourgainBrezis2007, DG,DG2,GRV,HS, HRS, LanzaniStein, RaitaSpector, RSS, Spector-VanSchaftingen-2018,VS,VS2,VS2a,VS3,VS4,SSVS}.

Inequalities in the spirit of \eqref{may1} hold in a more general setting than Lebesgue spaces. In a general form, they read
\begin{equation}\label{sobolevnablak}
 \|u\|_{Y(\rn, \R^\ell)} \leq c \|\nabla ^ku\|_{X(\rn, \R^{\ell\times n^k})},
\end{equation}
where $X(\rn, \R^{\ell\times n^k})$ and $Y(\rn, \R^\ell)$ are normed spaces. Diverse techniques have been introduced over the years for the proof of inequalities as in \eqref{sobolevnablak}
for specific classes of function spaces $X(\rn, \R^{\ell\times n^k})$ and $Y(\rn, \R^\ell)$. In particular, for rearrangement-invariant spaces, the inequality  \eqref{sobolevnablak} is known to be equivalent 
to the considerably simpler one-dimensional Hardy type inequality:
\begin{equation}\label{bound1k}
\bigg\|\int_s^\infty r^{-1+\frac k n}f(r)\, dr \bigg\|_{Y(0,\infty)} \leq c \|f\|_{X(0, \infty)}
\end{equation}
 for  $f \in X(0,\infty)$. 
Equivalence principles of this kind, for $k=1$, rest upon properties of symmetrization. Their use in the detection of optimal constants in Sobolev inequalities can be traced back to such contributions as \cite{Aubin, Moser, Talenti}.
 The result for arbitrary $k$ was achieved in \cite{KermanPick} for functions defined on bounded Lipschitz domains and extended, via a different method, to more general frameworks in \cite{cianchi-pick-slavikova, CPS_Frostman}. The case of functions defined on $\rn$ can be found in \cite{mihula}. 

As an inequality of the same nature as \eqref{korn} does not hold if  $L^p$ is replaced with an arbitrary rearrangement invariant space $X$ -- this typically happens for spaces $X$ \lq\lq close" to $L^1$ or $L^\infty$ -- the Sobolev seminorm in $X$ built upon a generic elliptic canceling operator $\mathcal A_k(D)$ is weaker than that built upon $\nabla^k$. 
 Nonetheless, we demonstrate in this paper that
any Sobolev type inequality for elliptic canceling operators $\mathcal A_k(D)$ of the form
\begin{equation}\label{sobolev1intro}
 \|u\|_{Y(\rn, \R^\ell)} \leq c \|\mathcal A_k(D)u\|_{X(\rn, \R^m)},
\end{equation}
with $1\leq k<n$, is  equivalent to \eqref{bound1k}, and hence also to
\eqref{sobolevnablak}.  

In summary, we have that,
if $\mathcal A_k(D)$ is an elliptic canceling operator, with $1\leq k<n$, and
  $X(\rn, \R^m)$ and  $Y(\rn, \R^\ell)$ are rearrangement-invariant spaces,  
then
\begin{verse}
  \emph{\lq\lq The Sobolev inequality   \eqref{sobolevnablak} for  $\nabla ^k$, the Sobolev inequality \eqref{sobolev1intro}  for $\mathcal A_k(D)$,  and  the one-dimensional inequality \eqref{bound1k} are equivalent."}
\end{verse}
This is the content of Theorem \ref{characterization}, which, in a sense, can be regarded as a version for arbitrary rearrangement-invariant spaces of the Bourgain-Brezis and Van Schaftingen $L^1$-results. In this theorem, as well as in all  Sobolev type inequalities  considered in the present paper, the differential operators are assumed to have order $k<n$. The regime $k\geq n$ typically pertains to the realm of Sobolev embeddings into spaces of continuous functions and functions endowed with a certain degree of smoothness, depending on $k$. This is well-known for the standard operator
$\nabla ^k$. An analysis of this kind of embeddings for more general homogeneous elliptic canceling operators is an interesting question, which, however, does not fall within the scope of this work.
Let us mention that results in the same spirit in the special case of the symmetric gradient operator were obtained in \cite{Breit-Cianchi}, via different ad hoc techniques.

Owing to Theorem \ref{characterization}, in Theorem \ref{sobolev-opt-canc} the optimal rearrangement-invariant target space  $Y(\rn, \R^\ell)$ in the inequality \eqref{sobolev1intro} for a given domain space $X(\rn, \R^m)$ is characterized.  Via this result, sharp Sobolev embeddings for elliptic canceling operators for diverse classes of rearrangement-invariant spaces  are offered in Section \ref{S:sobolev}. 
It is  clear from the discussion above that the conclusions are most relevant for borderline spaces $X(\mathbb{R}^n, \mathbb R^m)$
 which are \lq\lq close" to $L^1(\mathbb{R}^n, \mathbb R^m)$.
  Otherwise,   the inequality   \eqref{sobolev1intro}  can be deduced 
 from \eqref{sobolevnablak}
via a version of \eqref{korn} in $X(\mathbb{R}^n, \mathbb R^m)$.  Parallel results for Sobolev type spaces defined on open subsets $\Omega \subset \rn$ with finite Lebesgue measure and consisting of functions vanishing on $\partial \Omega$ are also established.

One illustrative implementation of our general result involves a logarithmic perturbation of  the space $L^1(\mathbb{R}^n, \mathbb R^m)$. 
Given $1\leq k <n$, any $r\geq 0$, and any elliptic canceling operator $\mathcal A_k(D)$, we have that
\begin{align}
    \label{exlog-sob}
    \|u\|_{L^{\frac n{n-k}}(\log L)^{\frac {nr}{n-k}}(\rn, \R^\ell)} \leq c \|\mathcal A_k(D)u\|_{L^1(\log L)^{r}(\rn, \rM)}
\end{align}
 for some constant $c$ and every function $u$ decaying  near infinity, where $L^p(\log L)^\alpha(\rn, \R^\ell)$ denotes the Zygmund space, which agrees (up to equivalent norms) with the Orlicz space associated with a Young function of the form $t^p\log^\alpha (b+t)$, for a suitable constant $b>0$. Moreover, $L^{\frac n{n-k}}(\log L)^{\frac {nr}{n-k}}(\rn, \R^\ell)$ is the  optimal target in \eqref{exlog-sob} among all Orlicz spaces. 
This is a special instance of Example \ref{ex1}. 
The inequality \eqref{exlog-sob} with
$A_k(D)=\nabla ^k$
bears some analogy with the inequality \eqref{may1} for $p=1$, as the argument of Sobolev which dominates a function by a Riesz potential of its $k$th order weak derivatives is insufficient to obtain this result.   Such an inequality can be deduced
via the equivalence to \eqref{bound1k} with $X(0,\infty)=L^1(\log L)^{r}(0,\infty)$ and $Y(0,\infty)=L^{\frac n{n-k}}(\log L)^{\frac {nr}{n-k}}(0, \infty)$.  By contrast, none of the  methods  available until now  was   applicable to obtain \eqref{exlog-sob} for a general $k$th order homogeneous elliptic canceling operator $\mathcal A_k(D)$. 

The equivalence of the inequalities  \eqref{bound1k} and \eqref{sobolev1intro} also enables one to recover, via a unified argument which applies for every $p\in [1, \frac nk)$, the inequality 
\begin{align}
    \label{sob-cancelp'}
 \|u\|_{L^{\frac{np}{n-kp},p}(\rn, \mathbb{R}^\ell)}\leq c \|\mathcal A_k(D) u\|_{L^p(\rn, \R^m)}
\end{align}
 for any elliptic canceling operator $\mathcal A_k(D)$, 
for some constant $c$,  and for every $u$ decaying near infinity. A noticeable feature of \eqref{sob-cancelp'} is in that the Lorentz space $L^{\frac{np}{n-kp},p}(\rn, \mathbb{R}^\ell)$ is optimal (smallest possible) among all rearrangement-invariant spaces.
 This inequality for $\mathcal A_k(D)=\nabla ^k$ is classical \cite{oneil, peetre}, and hence, when $1<p<\frac nk$, it follows for general $\mathcal A_k(D)$ via \eqref{korn}. 
For $p=1$ it
 seems not 
 to appear explicitly in the literature, but it can be deduced as a consequence of \cite[Theorem 3]{Stolyarov}, the relationship between Besov-Lorentz and Lorentz spaces discussed in \cite[Paragraph after Theorem 2]{Stolyarov},
and standard bounds for singular integral operators on Lorentz spaces.

Our approach combines a representation formula for $k$-times weakly differentiable functions in terms of  linear operators, enjoying special properties, applied to $\mathcal A_k(D)u$ with an estimate in rearrangement form, for the relevant operators. The representation formula consists of subsequent compositions of linear combinations of $\mathcal A_k(D)u$ with ad hoc finite-dimensional linear operators,  singular integral operators, and  $k$-th order Riesz potentials. The main feature of such a formula is that the finite-dimensional operators can be chosen in such a way that their composition with $\mathcal A_k(D)u$ results in $k'$-th order divergence free vector fields, for suitable $k'$. The rearrangement estimate amounts to an inequality between (integrals of) the decreasing rearrangement of  divergence free vector fields and the rearrangement of their composition with singular integral operators and Riesz potentials.
This rearrangement estimate is the content of Theorem \ref{K-CZ}. Its proof rests upon the computation of 
the  $K$-functional for
 the couple $(L^1(\rn, \rM),L^{p,q}(\rn, \rM))$, restricted to  divergence-free vector fields. This piece  of information, of independent interest,  is provided by Theorem \ref{lemma2-higherorder}, which   complements results from \cite{Bourgain,Pisier}. Its proof in turn requires a precise analysis of   mapping properties of an arbitrary order   Helmholz projection singular integral operator,
accomplished in Lemma \ref{lemma1-higherorder}.

\section{Background}\label{back}

  In this section we collect basic definitions and properties concerning functions and function spaces playing a role in the paper.
  
Throughout, the relation $\lq\lq \lesssim "$ between two positive expressions means that the former is bounded by the latter, up to a positive multiplicative constant depending on quantities to be specified.
The relations  $\lq\lq \gtrsim "$ and $\lq\lq \approx "$ are defined accordingly. 

We denote by $|E|$ the Lebesgue measure of a set $E\subset \rn$. Assume that $\Omega $ is a measurable subset of $\rn$, with $n \in \mathbb N$, and let $m \in \mathbb N$.  The notation 
$\M(\Omega, \mathbb R^m)$ is adopted for
 the space of all Lebesgue-measurable functions
$F : \Omega \to \mathbb R^m$. If $m=1$, we shall simply write $\M(\Omega)$. An analogous convention will be adopted for other function spaces. 
 We  also define $\Mpl(\Omega)=\{f\in\M(\Omega)\colon f\geq 0 \
\textup{a.e. in}\ \Omega\}$.

The \emph{decreasing rearrangement}  $F^{\ast}:[0, \infty )\to
[0,\infty]$ of a function $F \in \M(\Omega, \mathbb R^m)$ is   defined as
\begin{equation*}
F^{\ast}(s) = \inf \{t\geq 0: |\{x\in \Omega : |F(x)|>t \}|\leq s \}
\qquad \hbox{for $s \in [0,\infty)$}.
\end{equation*}
The function  $F^{**}: (0, \infty ) \to [0, \infty )$ is given by
\begin{equation}\label{c15}
F^{**}(s)=\frac{1}{s}\int_0^s F^*(r)\,dr \qquad \hbox{for $s>0$}.
\end{equation}
Notice that
\begin{align}
    \label{aug11}
    F^{**}(s)= \frac 1s \sup\bigg\{ \int_E|F|\,dx: E\subset \Omega, |E|=s\bigg\}.
\end{align}
The Hardy-Littlewood inequality tells us that
\begin{equation}\label{HL}
\int_{\Omega}|F(x)| |G(x)|dx\leq \int_0^\infty F^*(s) G^*(s)\, ds
\end{equation}
for $F, G\in \mathcal M (\Omega, \rM)$.
\\ Let $L\in (0,\infty]$. A functional
 $\|\cdot\|_{X(0,L)}{:\Mpl(0,L)\to[0,\infty]}$ is called a
\textit{function norm} if, for all functions $f, g\in \Mpl(0,L)$, all sequences
$\{f_k\} \subset {\Mpl(0,L)}$, and every $\lambda \geq 0$:
\begin{itemize}
\item[(P1)]\quad $\|f\|_{X(0,L)}=0$ if and only if $f=0$ a.e.;
$\|\lambda f\|_{X(0,L)}= \lambda \|f\|_{X(0,L)}$; \par\noindent \quad
$\|f+g\|_{X(0,L)}\leq \|f\|_{X(0,L)}+ \|g\|_{X(0,L)}$;
\item[(P2)]\quad $ f \le g$ a.e.\  implies $\|f\|_{X(0,L)}
\le \|g\|_{X(0,L)}$;
\item[(P3)]\quad $f_k \nearrow f$ a.e.\
implies $\|f_k\|_{X(0,L)} \nearrow \|f\|_{X(0,L)}$;
\item[(P4)]\quad $\|\chi _E\|_{X(0,L)}<\infty$ if $|E| < \infty$;
\item[(P5)]\quad if $|E|< \infty$, then there exists a constant
 $c$, depending on $E$ and $X(0,L)$, such that \\   $\int_E f(s)\,ds \le c
\|f\|_{X(0,L)}$.
\end{itemize}
Here, $E$ denotes a measurable set in $(0,L)$, and  $\chi_E$ stands for its characteristic function.
If, in addition,
\begin{itemize}
\item[(P6)]\quad $\|f\|_{X(0,L)} = \|g\|_{X(0,L)}$ whenever $f\sp* = g\sp *$,
\end{itemize}
we say that $\|\cdot\|_{X(0,L)}$ is a
\textit{rearrangement-invariant function norm}.
\\
The \textit{associate function norm}  $\|\cdot\|_{X'(0,L)}$ of a function norm $\|\cdot\|_{X(0,L)}$ is  defined as
$$
\|f\|_{X'(0,L)}=\sup_{\begin{tiny}
                        \begin{array}{c}
                       {g\in{\Mpl(0,L)}}\\
                        \|g\|_{X(0,L)}\leq 1
                        \end{array}
                      \end{tiny}}
\int_0^{L}f(s)g(s)ds
$$
for $ f\in\Mpl(0,L)$.

Let $\Omega$ be a measurable set in $\rn$, and let
 $\|\cdot\|_{X(0,|\Omega|)}$ be   a rearrangement-invariant function norm.  Then the space $X(\Omega, \rM)$ is
defined as the collection of all  functions  $F \in\M(\Omega, \rM)$
such that the quantity
\begin{equation}\label{norm}
\|F\|_{X(\Omega, \rM)}= \|F^*\|_{X(0,|\Omega|)} 
%
\end{equation}
is finite. The space $X(\Omega, \rM)$ is a Banach space, endowed
with the norm given by \eqref{norm}. 
 The space $X(0,|\Omega|)$ is called the
\textit{representation space} of $X(\Omega, \rM)$.
\\
The \textit{associate
space}  $X'(\Omega, \rM)$
of   a~rearrangement-invariant~space $X(\Omega, \rM)$ is
 the
rearrangement-invariant space   built upon the
function norm $\|\cdot\|_{X'(0,|\Omega|)}$.
\\
The \textit{H\"older type inequality}
\begin{equation}\label{holder}
\int_{\Omega}|F| |G|dx\leq\|F\|_{X(\Omega, \rM)}\|G\|_{X'(\Omega, \rM)}
\end{equation}
holds  for every $F\in X(\Omega, \rM)$ and $G \in X'(\Omega, \rM)$.
\\ The norm $\|\cdot\|_{X(\Omega, \rM)}$ is said to be absolutely continuous if for every $F\in X(\Omega, \rM)$ and every sequence of sets $\{E_k\}$, with $E_k\subset \Omega$ and $E_k \to \emptyset$, one has that $$\|F\chi_{E_k}\|_{X(\Omega, \rM)}\to 0.$$
 Given $F, G \in \mathcal M(\Omega, \rM)$, Hardy's lemma ensures that
\begin{align}
    \label{hardy}
    \text{if\,\, $F^{**}(s)\leq G^{**}(s)$,\,\, then \,\,$\|F\|_{X(\Omega, \rM)}\leq \|G\|_{X(\Omega, \rM)}$}
\end{align}
for every rearrangement-invariant space $X(\Omega, \rM)$.
\\
Let $X(\Omega, \rM)$ and $Y(\Omega, \rM)$ be rearrangement-invariant\
spaces. We write $X(\Omega, \rM) \to Y(\Omega, \rM)$ to denote that
$X(\Omega, \rM)$ is continuously embedded into $Y(\Omega, \rM)$, in the sense that
there exists a  constant $c$ such that
$\|F\|_{Y(\Omega, \rM)}\leq c\|F\|_{ X(\Omega, \rM)}$ for every $F\in X(\Omega, \rM)$.
\\
Assume that $X(\rn, \rM)$ is a rearrangement-invariant space. Then
\begin{equation}\label{dec2}
{L^1(\rn, \rM)\cap L^\infty(\rn, \rM)} \to
X(\rn, \rM) \to {L^1(\rn, \rM)+L^\infty(\rn, \rM)}.
\end{equation}
If  $|\Omega|< \infty$, then
\begin{equation}\label{l1linf}
L^\infty (\Omega, \rM) \to X(\Omega, \rM) \to L^1(\Omega, \rM)
\end{equation}
for every rearrangement-invariant space
$X(\Omega, \rM)$.
\\ 
Rearrangement-invariant spaces $X(\Omega, \rM)$ defined on a measurable set $\Omega \subset \rn$ can be extended to rearrangement-invariant spaces
defined on the whole of $\rn$ in a   canonical way as follows. Define the  function norm 
 $\|\cdot\|_{X^e(0, \infty)}$ as 
\begin{equation}\label{sep25}
\|f\|_{X^e(0, \infty)} = \|f^*\|_{X(0, |\Omega|)}
\end{equation}
for $f \in \mathcal M_+(0,\infty)$.
We denote by $X^e(\R^n, \rM)$ the rearrangement-invariant space associated with the function norm $\|\cdot\|_{X^e(0, \infty)}$. One has that 
\begin{equation}\label{feb83}
\|F\|_{X^e(\R^n, \rM)}=  \|F^*\|_{X(0, |\Omega|)}
\end{equation}
for every $F \in \mathcal M(\R^n, \rM)$. In particular, if $F =0$ a.e. in $\rn \setminus \Omega$, then 
\begin{equation}\label{extbis}
\|F\|_{X^e(\R^n, \rM)}=  \|F\|_{X(\Omega, \rM)}.
\end{equation}

Let $\{\rho_h\}$ be a family of smooth mollifiers, namely $\rho_h \in C^\infty_c(B_{1/h})$,   $\rho_h \geq 0$, and $\int_{\mathbb R^n}\rho_h (x)\,dx =1$ for $h \in \N$. Here, $B_R$ denotes the ball, centered at $0$, with radius $R$.
Given $F \in L^1_{\rm loc}(\mathbb R^n, \rM)$, set
$F_h = F *\rho_h$. Then, \begin{equation}\label{nov1}
\int_0^s F_h^* (\tau) \, d\tau \leq \int_0^s F^* (\tau) \, d\tau \qquad \text{for $s \geq 0$,}
\end{equation}
 for every $h \in \N$, and 
hence
\begin{equation}\label{nov2}
\|F_h\|_{X(\mathbb R^n, \rM)} \leq \|F\|_{X(\mathbb R^n, \rM)}
\end{equation}
for every  rearrangement-invariant space $X(\mathbb R^n, \rM)$ and every  $h\in \N$.

Given a couple of normed spaces $(Z_0, Z_1)$, which are both continuously embedded into some Hausdorff vector space, the $K$-functional is defined, for each $\zeta \in Z_0 + Z_1$,  as 
\begin{align}\label{Kfunct}
K(\zeta, t; Z_0, Z_1) = \inf \big\{  \|\zeta_0\|_{Z_0} + t \|\zeta_1\|_{Z_1}: \zeta = \zeta_0 + \zeta_1, \, \zeta_0 \in Z_0, \, \zeta_1 \in Z_1\big\} \quad \text{for $t>0$.}
\end{align}
Let $n, m\in \mathbb N$ and let $\|\cdot \|_{X(0,\infty)}$ and  $\|\cdot \|_{Y(0,\infty)}$ be rearrangement-invariant function norms.    
Then 
\begin{equation}\label{K-v1}
K(|F|, t; X(\rn), Y(\rn)) =K(F, t; X(\rn, \mathbb R^m), Y(\rn, \mathbb R^m)) 
\quad \text{for $t>0$,}
\end{equation}
for $F\in X(\rn, \mathbb R^m) +Y(\rn, \mathbb R^m)$. See \cite[Lemma 7.3]{Breit-Cianchi} for a proof.

Given $p\in [1,\infty)$ and $q\in [1,\infty]$, the Lorentz space $L^{(p,q)}(\Omega, \rM)$ is the rearrangement invariant space associated with the function norm given by
\begin{align}
    \label{lorentz}
    \|f\|_{L^{(p,q)}(0,|\Omega|)}= 
    \big\|s^{1/p-1/q}f^{**}(s)\|_{L^q(0,|\Omega|)}
%
\end{align}
for $f\in \Mpl (0,|\Omega|)$. 
The functional $\|\cdot\|_{L^{p,q}(\Omega, \rM)}$ and the space $L^{p,q}(\Omega, \rM)$ are defined analogously, by replacing $f^{**}$ with $f^*$ in \eqref{lorentz}. 
%
    This functional is a norm if $1\leq q\leq p$, and  is equivalent to the norm $\|\cdot\|_{L^{(p,q)}(\Omega, \rM)}$ if $p\in (1,\infty)$.
    \\
    The family of Lorentz spaces extends that of Lebesgue spaces, since $L^{(p,p)}(\Omega, \rM)= L^p(\Omega, \rM)$, up to equivalent norms, for $p\in (1,\infty)$, and $L^{p,p}(\Omega, \rM)= L^p(\Omega, \rM)$ for $p\in [1,\infty)$.
    \\ If $q<\infty$, then the norm $\|\cdot\|_{L^{(p,q)}(\Omega, \rM)}$ is absolutely continuous.
    \\ 
 For $p\in (1,\infty)$, the H\"older inequality in Lorentz spaces reads
\begin{align}
    \label{holderlor}
    \int_{\Omega} |F||G|  \, dx \leq  \|F\|_{L^{p,q}(\Omega, \rM)} \|G\|_{L^{p',q'}(\Omega, \rM)}
\end{align}
 for  $F\in L^{p,q}(\Omega, \rM)$ and $G\in L^{p',q'}(\Omega, \rM)$. Here, $p'$ and $q'$ denote the H\"older conjugates of $p$ and $q$.
\\ 
Moreover, if $p\in (1,\infty)$, then there exists a constant $c$ such that
\begin{align}
    \label{aug3}
\|F\|_{L^{p,q}(\rn, \rM)} \leq c \|F\|_{L^1(\rn, \rM)}^{\frac 1p} \|F\|_{L^\infty(\rn, \rM)}^{\frac 1{p'}}
\end{align}
for  $F\in {L^1(\rn, \rM)\cap L^\infty(\rn, \rM)}$.
\\
 The Lorentz-Zygmund spaces $L^{(p,q, r)}(\Omega, \rM)$ provide a generalization of the Lorentz spaces. Their norm is built upon the function norm given, for $p\in [1, \infty)$, $q\in [1, \infty]$, and $r \in \mathbb R$,   by 
\begin{equation}\label{LZ}
\|f\|_{L\sp{(p,q,r)}(0,|\Omega|)}=
\left\|s\sp{\frac{1}{p}-\frac{1}{q}}(1+\log_+ (\tfrac{1}{s}))^r f^{**}(s)\right\|_{L\sp q(0,|\Omega|)}
\end{equation}
for  $f \in {\Mpl(0,|\Omega|)}$. Here, $\log_+$ stands for the positive part of $\log$.  
A replacement of $f^{**}$ with $f^*$ in \eqref{LZ} leads to the functional  $\|\cdot\|_{L^{p,q,r}(\Omega, \rM)}$ and a corresponding space $L^{p,q,r}(\Omega, \rM)$. The new functional is equivalent to $\|\cdot\|_{L^{(p,q,r)}(\Omega, \rM)}$ if $p\in (1,\infty)$.
\\ In a few borderline inequalities, the 
generalized Lorentz-Zygmund spaces will also emerge. They are built upon the functional defined as 
\begin{equation}\label{GLZ}
\|f\|_{L^{p,q,r,\varrho}(0,|\Omega|)}=
\left\|s\sp{\frac{1}{p}-\frac 1q}(1+\log_+ (\tfrac{1}{s}))^{r} (1+\log_+ (1 + \log_+ \tfrac{1}{s})))^{\varrho}f^{*}(s)\right\|_{L\sp q(0,|\Omega|)}
\end{equation}
for  $f \in {\Mpl(0,|\Omega|)}$, where $p,q,r$ are as above and $\varrho \in \mathbb R$.
%
\\
The  Orlicz spaces extend the Lebesgue spaces in a different direction. Their definition
rests on that of a \textit{Young
function}, namely a left-continuous convex function  from  $[0, \infty )$ into  $[0, \infty ]$ that vanishes
at $0$ and is not constant in $(0, \infty)$. Any Young function $A$ can be represented as
\begin{equation}\label{A}
A(t) = \int _0^t a(s)\, ds \quad \quad \textup{for}\ t\in[0,\infty),
\end{equation}
 for a non-decreasing, left-continuous function $a: [0, \infty )\to
[0, \infty ]$, which is neither identically equal to $0$, nor to infinity. 
\\
Two Young functions $A$ and $B$ are said to be equivalent   globally/near infinity/near zero if there exists a positive constant $c$ such that
\begin{align}
    \label{2024-210}
    A(t/c) \leq B(t) \leq A(ct)
\end{align}
for $t\geq 0$/ $t\geq t_0$ for some $t_0>0$/$0\leq t\leq t_0$ for some $t_0>0$, respectively. We shall write
\begin{align}
\label{2024-211} 
A \simeq B
\end{align}
to denote the  equivalence between $A$ and $B$ in the sense of \eqref{2024-210}. Note that, for Young functions $A$ and $B$, the relation \eqref{2024-211} implies  that $A \approx B$, but the converse is not true.
\\ The \textit{Orlicz space}
$L\sp A(\Omega, \rM)$ is the rearrangement-invariant  space defined via
the  
\textit{Luxemburg function norm} given
by
\begin{equation}\label{lux}
\|f\|_{L^A(0,|\Omega|)}= \inf \left\{ \lambda >0 :  \int_{0}^{|\Omega|}A \left(
\frac{f(t)}{\lambda} \right) dt \leq 1 \right\}
\end{equation}
for  $f \in {\Mpl(0,|\Omega)}$. The alternate notation $A(L) 
(\Omega, \rM)$ will also be employed when convenient for explicit choices of the function $A$.
\\
The  norms $\|\cdot \|_{L^A
(\Omega, \rM)}$ and $\|\cdot \|_{L^B (\Omega, \rM)}$ are equivalent if and only if either $|\Omega|<\infty$ and
$A$ and $B$ are equivalent near infinity, or $|\Omega|=\infty$ and $A$ and $B$ are equivalent globally.
\\ Besides the Lebesgue spaces $L^p(\Omega, \rM)$, corresponding to the choice $A(t)=t^p$ if $p \in [1, \infty)$ and $A(t)=\chi_{(1, \infty)}\infty$ if $p=\infty$, the Zygmund spaces and the exponential type spaces are customary instances of Orlicz spaces. The Zygmund spaces $L^p(\log L)^r(\Omega, \rM)$ are associated with Young functions of the form $A(t)= t^p(\log (c+t))^r$, where either $p>1$ and $r \in \mathbb R$, or $p=1$ and $r\geq 0$, and $c$ is large enough for $A$ to be convex. If $|\Omega|<\infty$,  one has that  \begin{align}
    \label{eqZYg}
    \|f\|_{L^p(\log L)^r(0, |\Omega|)} \approx \big\|(1+\log_+ (\tfrac{1}{s}))^\frac{r}{p}f^*(s)\big\|_{L^p(0, |\Omega|)}
\end{align}
for $f \in \Mpl (0, |\Omega|)$, up to multiplicative constants independent of $f$ -- see \cite[Lemma 6.12, Chapter 4]{BennettSharpley}. Hence, if $p\in (1,\infty)$, then $L^p(\log L)^r(\Omega, \rM)= L^{p,p, \frac rp}(\Omega, \rM)$, up to equivalent norms.
The exponential spaces $\exp L^r (\Omega, \rM)$, for $r>0$, are built upon Young functions  $A(t)\simeq   e^{t^r}-1$ near infinity.


\par  The
family of Orlicz-Lorentz spaces encompasses various instances of Orlicz,  Lorentz, and Lorentz-Zygmund spaces. A special 
class of Orlicz-Lorentz spaces, which comes into play in our applications, is defined as follows.
Given  a Young function $A$ and a number $q\in \R$, we denote by
 $L(A,q)(\Omega, \rM)$ the \emph{Orlicz-Lorentz space}  defined via the functional
\begin{equation}\label{sep35}
	\|f\|_{L(A, q)(0,|\Omega|)}
		= \big\|r^{-\frac{1}{q}}f^{*}(r)\big\|_{L^A(0,|\Omega|)}
\end{equation}
for $f\in \Mpl (0,|\Omega|)$.
 Under proper assumptions on $A$ and $q$, this functional is actually a function norm. This is the case, for instance, if  $q>1$ and
\begin{equation}\label{sep36}
\int^\infty \frac{A(t)}{t^{1+q}}\, dt < \infty\,,
\end{equation}
see \cite[Proposition 2.1]{cianchi_ibero}.  

  \subsection{Sobolev type spaces}
Let $\Omega$ be an open set in $\rn$, with $n \geq 2$,  and let $\|\cdot \|_{X(0, |\Omega|)}$ be a rearrangement-invariant function norm. Given $k, m \in \N$ and a $k$-th order homogeneous differential operator,  with constant coefficients, $\mathcal A_k(D)$ mapping $\mathbb{R}^\ell$-valued functions to $\mathbb{R}^m$-valued functions, the homogeneous Sobolev space associated with $\|\cdot \|_{X(0, |\Omega|)}$ is defined as 
\begin{equation}\label{XA}
V^{\mathcal A_k}X(\Omega, \R^\ell)= \big\{ u: \Omega \to \mathbb{R}^\ell:  \,\,\text{$u$ is $(k-1)$-times weakly differentiable and }\,   \mathcal A_k(D) u \in X(\Omega, \rM)\big\}.
\end{equation}

Plainly, polynomials of degree not exceeding $k-1$ have to be excluded when considering   Sobolev type inequalities  for functions in $V^{\mathcal A_k}X(\Omega, \mathbb{R}^\ell)$. 
When $\Omega =\rn$,
this is accomplished by restricting the relevant inequalities to 
 a subspace $V^{\mathcal A_k}_{\circ} X(\rn, \mathbb{R}^\ell)$ consisting of those functions which decay to zero, together with their derivatives up to the order $k-1$, in a suitable weak  sense. Specifically, this subspace is defined as:
\begin{equation}\label{XAD}
V^{\mathcal A_k}_{\circ} X(\rn, \mathbb{R}^\ell) = \Bigg\{ u\in V^{\mathcal A_k}X(\rn, \mathbb{R}^\ell):\, \lim_{R\to \infty} R^j \dashint_{B_{2R} \setminus B_R} |\nabla ^ju|\, dx =0, \,\, j=0, \dots ,  k-1\Bigg\},
\end{equation}
 where 
 $\dashint$ stands for an averaged integral and $B_R$ denotes the ball centered at $0$ with radius $R$. The space $C^\infty_\circ (\rn, \mathbb{R}^\ell)$ is defined accordingly.
 \\
 Of course, any function in $V^{\mathcal A_k}X(\rn, \mathbb{R}^\ell)$ with bounded support belongs to $V^{\mathcal A_k}_{\circ} X(\rn, \mathbb{R}^\ell)$. On the other hand, if $u\in V^{\mathcal A_k}_{\circ} X(\rn, \mathbb{R}^\ell)$, then 
 \begin{align}
     \label{levels}
     |\{ |\nabla ^ju|>t\}|<\infty 
 \quad \text{for $t>0$,}
 \end{align}
 for  $j=0, \dots , k-1$.
 \\ Sobolev inequalities on open sets $\Omega$ such that $|\Omega|<\infty$ will also be considered. Admissible functions in the relevant $k$-th order inequalities  will be required to vanish, together with their derivatives up to the order $k-1$, on $\partial \Omega$. This means that the functions in question belong to the space
\begin{equation}\label{X0}
V^{\mathcal A_k}_0 X(\Omega, \mathbb{R}^\ell) = \big\{ u\in V^{\mathcal A_k}X(\Omega, \mathbb{R}^\ell): \text{the extension of $u$ by $0$ outside $\Omega$ belongs to $V^{\mathcal A_k} X(\rn, \mathbb{R}^\ell)$}\big\}.
\end{equation}
 In the special case when $\mathcal A_k(D)=\nabla ^k$, we simply denote the space $V^{\mathcal A_k}X(\Omega, \mathbb{R}^\ell)$ by $V^{k}X(\Omega, \mathbb{R}^\ell)$. A parallel convention is adopted for the spaces  $V^{\mathcal A_k}_{\circ} X(\rn, \mathbb{R}^\ell)$ and $V^{\mathcal A_k}_0X(\Omega, \mathbb{R}^\ell)$.

\section{Sobolev inequalities for elliptic canceling operators}\label{S:sobolev}

The characterization of rearrangement-invariant norms supporting Sobolev inequalities for elliptic canceling   operators in terms of a one-dimensional Hardy type inequality is the main result of this paper and is offered below. As mentioned above, such a characterization only depends on the order of the differential operator.  It is  hence the same for all the operators in this class, including  the classical full $k$-th order gradient. As a consequence, any embedding of Sobolev type for the standard $k$-th order gradient, involving rearrangement-invariant norms, directly translates into an embedding with the same norms for any elliptic canceling  operator of the same order.

 \smallskip
 
Elliptic differential operators can be defined as follows.

\begin{definitionalph}[{\bf Elliptic operator}]
    \label{ell-op} {\rm  Let $n, \ell, m \in \mathbb N$.
     A linear homogeneous $k$-th order constant coefficient differential operator
$\mathcal A_k(D)$ mapping $\mathbb{R}^\ell$-valued functions to $\mathbb{R}^m$-valued functions
is said to be  elliptic if the linear  map $\mathcal A_k(\xi):\mathbb R^{\ell}\rightarrow\mathbb R^m$ is injective for every $\xi\in\mathbb R^n\setminus \{0\}$. Here,   $\mathcal A_k(\xi)$ denotes the symbol map of $\mathcal A_k(D)$ in terms of Fourier transforms.}
\end{definitionalph}

The  notion of canceling differential operator adopted in this paper is patterned on
 \cite[Definition 1.2]{VS3}. 

 \begin{definitionalph}[{\bf Canceling operator}]
    \label{canc-op}{\rm
     Let $n, \ell, m \in \mathbb N$. A linear homogeneous $k$-th order constant coefficient  differential operator $\mathcal A_k(D)$ mapping $\mathbb{R}^\ell$-valued functions to $\mathbb{R}^m$-valued functions
  is said to be  canceling if 
\begin{align*}
 \bigcap_{\xi\in\mathbb{R}^n \setminus \{0\}}\ \mathcal A_k(\xi)[\mathbb{R}^\ell]=\{0\}.
\end{align*}}
\end{definitionalph}

The symmetric gradient is a classical instance of first-order elliptic canceling operator. It is defined as
$$\mathcal E u = \frac{\nabla u + (\nabla u)^T}2$$
at functions $u: \rn \to \rn$, and
has a role in  modeling diverse physical 
phenomena, such as the theory 
of (generalized) Newtonian 
fluids and the  theories of plasticity and nonlinear
elasticity. Another important
example of an operator of this class is the trace-free version of $\mathcal E$ for $n\geq 3$, also 
called the deviatoric part of the symmetric gradient. It is given by 
$$\mathcal E^Du = \mathcal E u - \frac{{\rm tr}(\mathcal Eu)}{n}I.$$
This operator comes into play   in the analysis of mathematical models for compressible 
fluids,  in the
Cauchy formulation of the Einstein gravitational field equations, in the Cosserat theory of elasticity. \\ Further samples of less conventional elliptic canceling operators can be found in \cite[Section 6]{VS3}.


\begin{theorem}[{\bf Reduction principle for Sobolev inequalities for canceling operators in $\rn$}]\label{characterization} 
Let $n, m, \ell, k \in \N$, with $n \geq 2$     and $1 \leq k <n$.  Assume that $\mathcal A_k(D)$ is a linear homogeneous
$k$-th order  elliptic canceling operator.
Let   $\|\cdot\|_{X(0,\infty)}$ and $\|\cdot\|_{Y(0,\infty)}$ be rearrangement-invariant function norms.
  The following facts are equivalent:
\\ (i) The embedding 
$V^{\mathcal A_k}_{\circ} X(\rn, \mathbb{R}^\ell) \to Y(\rn, \R^\ell)$ holds, i.e. 
there exists a constant $c_1$ such that
\begin{equation}\label{sobolev1}
\|u\|_{Y(\rn, \mathbb R^\ell)} \leq c_1 \|\mathcal A_k(D)u\|_{X(\rn, \rM)}
\end{equation}
 for every $u \in V^{\mathcal A_k}_{\circ} X(\rn, \mathbb{R}^\ell)$. \\ (ii) The embedding 
 $V^{k}_\circ X(\rn, \R^{\ell}) \to Y(\rn, \R^\ell)$ holds, i.e. 
 there exists a constant $c_2$ such that
\begin{equation}\label{sobolev2}
\|u\|_{Y(\rn, \R^\ell)} \leq c_2 \|\nabla ^ku\|_{X(\rn, \R^{\ell \times n^k})}
\end{equation}
 for every $u \in V^{k}_\circ X(\rn, \R^{\ell})$. 
 \\
 (iii)
There exists a constant $c_3$ such that
\begin{equation}\label{sobolev3}
\bigg\|\int_s^\infty r^{-1+\frac {k}{n}}f(r)\, dr \bigg\|_{Y(0,\infty)} \leq c_3 \|f\|_{X(0, \infty)}
\end{equation}
for every $f \in X(0,\infty)$.
\\ Moreover, the constant $c_1$ depends only on  $c_3$ and $\mathcal A_k(D)$,   while $c_3$ depends only on $c_2$,  $k$, $n$ and $m$.
\end{theorem}

The following description of the optimal rearrangement-invariant target space in Sobolev inequalities for elliptic canceling operators can be derived from Theorem \ref{characterization}, via \cite[Theorem~2.1]{mihula}.
 The optimal space in question, $X_k(\mathbb R^n, \mathbb R^\ell)$, is defined via its associate   space $X_k'(\mathbb R^n, \mathbb R^m)$.  The latter  is built upon the function norm defining   obeying
\begin{align}
    \label{may7}
    \|f\|_{X_
    k'(0, \infty)}= \|s^{\frac k n} f^{**}(s)\|_{X'(0,\infty)}
\end{align}
for $f\in \Mpl (0, \infty)$. Here, $\|\,\cdot \,\|_{X'(0,\infty)}$ denotes the function norm which defines the   associate space of $X(\mathbb R^n, \mathbb R^m)$. The right-hand side of \eqref{may7} is a rearrangement-invariant function norm provided that
\begin{equation}\label{may8}
    \|(1+r)^{-1+\frac{k}{n}}\|_{X'(0,\infty)}< \infty,
\end{equation}
see \cite[Theorem 4.4]{EMMP}.   Notice that the condition \eqref{may8} is necessary for an inequality of the form \eqref{sobolev3} to hold whatever the rearrangement-invariant space $Y(\rn, \rM)$ is. This follows analogously to \cite[Equation (2.2)]{EGP}.

\begin{theorem}[{\bf Optimal target space in Sobolev inequalities  in $\rn$}]\label{sobolev-opt-canc}   Let $n, m, \ell, k$, $\mathcal A_k(D)$ and \,  $\|\cdot\|_{X(0,\infty)}$ be as in Theorem \ref{characterization}.  
Assume, in addition, that   $\|\cdot\|_{X(0,\infty)}$ satisfies the assumption \eqref{may8}.
Then,   
 $$V^{\mathcal A_k}_{\circ} X(\rn, \mathbb R^\ell) \to X_k(\mathbb R^n, \mathbb R^\ell)$$ and
there exists a constant $c=c(n, m, k, \mathcal A_k(D))$ such that
\begin{equation}\label{sobolev-bound-opt}
\|u\|_{X_k(\mathbb R^n, \mathbb R^\ell)} \leq c \|\mathcal A_k(D)u\|_{X(\rn, \mathbb R^m)}
\end{equation} 
for every  $u \in V^{\mathcal A_k}_{\circ} X(\rn, \mathbb R^\ell)$.
\\ The assumption \eqref{may8} is necessary for an inequality of the form \eqref{sobolev1} to hold, whatever $Y(\mathbb R^n, \mathbb R^\ell)$ is. Moreover,
 the target space $X_k(\mathbb R^n, \mathbb R^\ell)$ is optimal (smallest possible) in \eqref{sobolev1} among all rearrangement-invariant spaces.
\end{theorem}

A version of Theorem \ref{characterization} for embeddings of the spaces $V^{\mathcal A_k}_0X(\Omega, \mathbb{R}^\ell)$ for domains $\Omega$ with finite measure reads as follows.

\begin{corollary}[{\bf Reduction principle for Sobolev inequalities  in domains}]
\label{characterization-omega} 
Let $n, m, \ell, k$ and $\mathcal A_k(D)$  be as in Theorem \ref{characterization}. 
Assume that $\Omega$ is an open set in $\rn$ with $|\Omega|<\infty$ and let
 $\|\cdot\|_{X(0,|\Omega|)}$ and $\|\cdot\|_{Y(0,|\Omega|)}$ be rearrangement-invariant function norms.
%
  The following facts are equivalent:
\\ (i) The embedding $V^{\mathcal A_k}_0 X(\Omega, \mathbb{R}^\ell) \to Y(\Omega, \R^\ell)$ holds, i.e. there exists a constant $c_1$ such that
\begin{equation}\label{sobolev1om}
\|u\|_{Y(\Omega, \R^\ell)} \leq c_1 \|\mathcal A_k(D)u\|_{X(\Omega, \rM)}
\end{equation}
 for every $u \in V^{\mathcal A_k}_0 X(\Omega, \mathbb{R}^\ell)$. 
\\ (ii) The embedding 
$V^{k}_0 X(\Omega, \R^{\ell}) \to Y(\Omega, \R^\ell)$ holds, i.e.
there exists a constant $c_2$ such that
\begin{equation}\label{sobolev2om}
\|u\|_{Y(\Omega, \R^\ell)} \leq c_2 \|\nabla ^ku\|_{X(\Omega, \R^{\ell \times n^k})}
\end{equation}
 for every $u \in V^{k}_0 X(\Omega, \R^{\ell})$.
 \\
 (iii)
There exists a constant $c_3$ such that
\begin{equation}\label{sobolev3om}
\bigg\|\int_s^{|\Omega|} r^{-1+\frac {k}{n}}f(r)\, dr \bigg\|_{Y(0,|\Omega|)} \leq c_3 \|f\|_{X(0, |\Omega|)}
\end{equation}
for every $f \in X(0,|\Omega|)$.
\\
Moreover, the constant $c_1$ depends only on $c_3$ and $\mathcal A_k(D)$,    while the constant $c_3$ depends only on $c_2$, $k$, $n$ and $m$.
\end{corollary}

Combining Corollary \ref{characterization-omega} with \cite[Proposition 5.2]{KermanPick} provides us with
the characterization of an optimal rearrangement-target space  for embeddings of any Sobolev space $V^{\mathcal A_k}_0X(\Omega, \mathbb{R}^\ell)$.
The optimal space in question is denoted by $X_k (\Omega, \R^m)$ and defined, in analogy with \eqref{may7}, as the rearrangement invariant space whose associate norm is given via the function norm
\begin{align}
    \label{may7-domain}
    \|f\|_{X_
    k'(0, |\Omega|)}= \|s^{\frac k n} f^{**}(s)\|_{X'(0,|\Omega|)}
\end{align}
for $f\in \Mpl (0, |\Omega|)$.
Notice that no additional assumption like \eqref{may8} is required in this case.

\begin{corollary}[{\bf Optimal target space in Sobolev inequalities  in domains}]\label{sobolev-opt-canc-domain}   Let $n, m, \ell, k$, $\mathcal A_k(D)$, $\Omega$,   $\|\cdot\|_{X(0,|\Omega|)}$, and $\|\cdot\|_{X_k(0,|\Omega|)}$ 
 be as in Corollary \ref{characterization-omega}.  
Then,    
$$V^{\mathcal A_k}_{0} X(\Omega, \mathbb{R}^\ell) \to X_k(\Omega, \mathbb R^\ell)$$  and
there exists a constant $c=c(n, m, k, \mathcal A_k(D))$ such that
\begin{equation}\label{sobolev-bound-opt-om}
\|u\|_{X_k(\Omega, \mathbb R^\ell)} \leq c \|\mathcal A_k(D)u\|_{X(\Omega, \mathbb R^m)}
\end{equation} 
for every  $u \in V^{\mathcal A_k}_{0} X(\Omega, \mathbb{R}^\ell)$.
\\ Moreover,
 the target space $X_k(\Omega, \mathbb R^\ell)$ is optimal (smallest possible) in \eqref{sobolev-bound-opt-om} among all rearrangement-invariant spaces.
\end{corollary}

The  theorems  above can be used  to derive a  number of new inequalities in families of rearrangement-invariant spaces. They include, for instance, Zygmund spaces and, more generally, Orlicz spaces and  Lorentz-Zygmund spaces.

\par We begin with Orlicz domain and target spaces. Given 
  a~Young function $A$ such that
\begin{equation}\label{conv0}
	\int_{0}\left(\frac{t}{A(t)}\right)^{\frac{k}{n-k}}\,dt < \infty,
\end{equation}	
let 
 $A_{\frac{n}{k}}$ be its Sobolev conjugate defined as
\begin{equation}\label{An}
A_{\frac{n}{k}} (t) = A(H^{-1}(t)) \quad \text{for $t\geq 0$,}
\end{equation}
where
\begin{equation}\label{H}
H(t) = \bigg(\int _0^t \bigg(\frac \tau{A(\tau)}\bigg)^{\frac
{k}{n-k}} d\tau\bigg)^{\frac {n-k}n} \quad \text{for $t \geq0$.}
\end{equation}
The Young function $A_{\frac{n}{k}}$  defines the optimal Orlicz target space for embeddings of  $k$-th order Orlicz-Sobolev spaces defined via the full $k$-th order gradient operator $\nabla ^k$.
This is shown in  \cite{cianchi_CPDE} (see also \cite{cianchi_IUMJ} for an alternate equivalent formulation) for $k=1$ and in \cite{cianchi_forum} for arbitrary $k$. Analogous results for fractional-order Orlicz-Sobolev spaces are established in \cite{ACPS, ACPS_NA}. The next result ensures that the same conclusion holds for any $k$-order elliptic canceling operator.

\begin{theorem}[{\bf Optimal Orlicz-Sobolev inequalities}]\label{sobolev-orlicz-canc} 
  Let $n,m,\ell,k$ and $\mathcal A_k(D)$ be as in Theorem \ref{characterization}.
\\ (i) Assume that $A$ is a Young function fulfilling the condition \eqref{conv0} with $\alpha =k$. Let   $A_{\frac nk}$ be the Young function defined as in \eqref{An}.
 Then, $$V^{\mathcal A_k}_\circ L^A(\rn, \mathbb{R}^\ell) \to L^{A_{\frac nk}}(\mathbb R^n, \mathbb R^\ell)$$ and
there exists a constant $c=c(\mathcal A_k(D))$ such that
\begin{equation}\label{bound-orlicz-canc}
\|u \|_{L^{A_{\frac nk}}(\mathbb R^n, \mathbb R^\ell)} \leq c \|\mathcal A_k(D)u\|_{L^A(\rn, \mathbb R^m)}
\end{equation}
for  every  $u \in V^{\mathcal A_k}_\circ L^A(\rn, \mathbb{R}^\ell)$.
\\ (ii) Assume that $\Omega$ is an open set in $\rn$ such that $|\Omega|<\infty$. Let   $A_{\frac nk}$ be a Young function defined as in \eqref{An} with $A$ modified, if necessary, near $0$ in such a way that the condition \eqref{conv0} is fulfilled. Then,
$$V^{\mathcal A_k}_0 L^A(\Omega, \mathbb{R}^\ell) \to L^{A_{\frac nk}}(\Omega, \mathbb R^\ell)$$
and
there exists a constant $c=c(A, \mathcal A_k(D), |\Omega|)$ such that
\begin{equation}\label{bound-orlicz-domain-canc}
\|u\|_{L^{A_{\frac nk}}(\Omega, \mathbb R^\ell)} \leq c \|\mathcal A_k(D)u\|_{L^A(\Omega, \mathbb R^m)}
\end{equation}
for  every  $u \in V^{\mathcal A_k}_0 L^A(\Omega, \mathbb{R}^\ell)$.
\\ In both inequalities \eqref{bound-orlicz-canc} and \eqref{bound-orlicz-domain-canc}, the target space is optimal (smallest possible) among all Orlicz spaces.
\end{theorem}

\begin{remark}
    \label{linfinity}
   {\rm  If $A$ grows  so fast   near infinity that
\begin{equation}\label{convinf}
	\int^\infty\left(\frac{t}{A(t)}\right)^{\frac{k}{n-k}}\,dt < \infty,
\end{equation}	
then $A_{\frac nk}(t)=\infty$ for large $t$. Hence, $L^{A_{\frac nk}}(\rn, \mathbb R^\ell)\to L^\infty (\rn, \mathbb R^\ell)$ and, if $|\Omega|<\infty$, then $L^{A_{\frac nk}}(\Omega, \mathbb R^\ell)= L^\infty (\Omega, \mathbb R^\ell)$, up to equivalent norms. Therefore, the inequality \eqref{bound-orlicz-canc}
implies that
\begin{equation}\label{bound-orlicz-inf}
\|u \|_{L^{\infty}(\mathbb R^n, \mathbb R^\ell)} \leq c \|\mathcal A_k(D)u\|_{L^A(\rn, \mathbb R^m)}
\end{equation}
for  every  $u \in V^{\mathcal A_k}_\circ L^A(\rn, \mathbb{R}^\ell)$. A parallel inequality for functions in $V^{\mathcal A_k}_\circ L^A(\Omega, \mathbb{R}^\ell)$ follows from \eqref{bound-orlicz-domain-canc}.}
\end{remark}

\begin{example}\label{ex1}{\rm
Consider a Young function $A$ such that
\begin{equation}\label{june3}
A(t) \,\,\simeq\,\, \begin{cases} t^{p_0} (\log \frac 1t)^{r_0} & \quad \text{near zero}
\\
t^p  (\log t)^r & \quad \text{near infinity,}
\end{cases}
\end{equation}
where either $p_0>1$ and $r_0 \in \R$, or $p_0=1$ and $r_0 \leq 0$, and either $p>1$ and $r \in \R$, or $p=1$ and $r \geq 0$.  
\\  The function $A$ satisfies the assumption \eqref{conv0} if
\begin{equation}\label{june4}
\text{either $1\leq p_0< \frac nk$ and $r_0$ is as above, or $p_0=\frac nk$ and $r_0 > \frac nk -1$.}
\end{equation}
\\ Theorem \ref{sobolev-orlicz-canc}  tells us that the inequality \eqref{bound-orlicz-canc} holds,  where
\begin{equation}\label{june5}
A_{\frac nk}(t)\,\,  \simeq\,\, \begin{cases} t^{\frac {n{p_0}}{n-k{p_0}}} (\log \frac 1t)^{\frac {nr_0}{n-k{p_0}}} & \quad \text{ if $1\leq {p_0}< \frac nk$ }
\\
e^{-t^{-\frac{n}{k(r_0 +1)-n}}} & \quad  \text{if ${p_0}=\frac nk$ and $r_0 > \frac nk -1$}
\end{cases} \quad \text{near zero,}
\end{equation}
and
\begin{equation}\label{dec256}
A_{\frac nk}(t) \,\, \simeq \,\, \begin{cases} t^{\frac {np}{n-k p}} (\log t)^{{\color{black}\frac {n r  }{n-k p}}} & \quad  \text{ if $1\leq p< \frac nk$ }
\\
e^{t^{\frac{n}{n-(r +1)k}}}&  \quad \text{if  $p=\frac nk$ and $r < \frac nk -1$}
\\
e^{e^{t^{\frac n{n-k}}}} &  \quad \text{if  $p=\frac nk$ and $r = \frac nk -1$}
\\
\infty &  \quad \text{otherwise}
\end{cases} \quad \text{near infinity.}
\end{equation} 
In particular, the choice $p_0=p=1$ and $r_0=0$ yields the inequality \eqref{exlog-sob}.  
\\ Analogous conclusions hold with regard to
the Sobolev inequality \eqref{bound-orlicz-domain-canc} on domains. However,  since  $|\Omega|<\infty$,
only the behaviors near infinity of $A$ and $A_{\frac nk}$ displayed above are relevant in this case. In particular, the assumption \eqref{june4} can be dropped.}
\end{example}

\begin{example}\label{ex2}{\rm
Let $A$ be a Young function such that
\begin{equation}\label{sep50}
A(t) \,\,\simeq\,\, \begin{cases} t^{p_0} (\log (\log \frac 1t))^{r_0} & \quad \text{near zero}
\\
t^p  (\log (\log t))^r & \quad \text{near infinity,}
\end{cases}
\end{equation}
where either $p_0>1$ and $r_0 \in \R$, or $p_0=1$ and $r_0 \leq 0$, and either $p>1$ and $r \in \R$ or $p=1$ and $r \geq 0$.  
\\  This function satisfies the assumption \eqref{conv0} if
\begin{equation}\label{sep51}
\text{$1\leq p_0< \frac nk$ and $r_0$ is as above.}
\end{equation}
\\ From Theorem \ref{sobolev-orlicz-canc}  we infer that the inequality \eqref{bound-orlicz-canc} holds,  with
\begin{equation}\label{sep52}
A_{\frac nk}(t)\, \simeq\, t^{\frac {n{p_0}}{n-k{p_0}}} (\log (\log \tfrac 1t))^{\frac {nr_0}{n-k{p_0}}} \quad \quad\text{near zero,}
\end{equation}
and
\begin{equation}\label{sep53}
A_{\frac nk}(t) \,\, \simeq \,\, \begin{cases} t^{\frac {np}{n-k p}} (\log (\log t))^{{\color{black}\frac {n r  }{n-k p}}} & \quad  \text{ if $1\leq p< \frac nk$ }
\\ \\
e^{t^{\frac{n}{n-k}}(\log t)^{\frac {rk}{n-k}}} &  \quad \,\,\text{if  $p=\frac nk$}
\end{cases} \quad \text{near infinity.}
\end{equation} 
For $p_0=p=1$ and $r_0=0$ this results in  the inequality 
\begin{align}
    \label{exloglog}
    \|u\|_{L^{\frac n{n-k}}(\log \log L)^{\frac {nr}{n-k}}(\rn,  \mathbb R^\ell)} \leq c \|\mathcal A_k(D)u\|_{L^1(\log \log L)^{r}(\rn, \rM)}
\end{align}
for some constant $c$ and  every $u \in V^{\mathcal A_k}_\circ L^1(\log \log L)^{r}(\rn, \mathbb{R}^\ell)$.
\\ Conclusions in the same spirit hold for the inequality \eqref{bound-orlicz-domain-canc}, where $|\Omega|<\infty$, with  simplifications analogous to those described in Example \ref{ex1}.}
\end{example}

\medskip

The conclusions of Theorem \ref{sobolev-orlicz-canc}
can still be improved if  Orlicz-Lorentz target spaces, defined as follows, are allowed.
 Let $A$ be  a Young function fulfilling the condition \eqref{conv0} and let $a: [0, \infty) \to [0, \infty)$ be the left-continuous function such that
\begin{align}
    \label{a}
    A(t) = \int_0^ta(\tau)\, d\tau \qquad \text{for $t\geq 0$.}
\end{align}
Denote by  $\widehat A$ the Young function given by
\begin{equation}\label{E:1}
	\widehat A (t)=\int_0^t\widehat a (\tau)\,d\tau\quad\text{for $t\geq 0$},
\end{equation}
where
\begin{equation}\label{E:2}
	{\widehat a\,}^{-1}(r) = \left(\int_{a^{-1}(r)}^{\infty}
		\left(\int_0^t\left(\frac{1}{a(\varrho)}\right)^{\frac{k}{n-k}}\,d\varrho\right)^{-\frac{n}{k}}\frac{dt}{a(t)^{\frac{n}{n-k}}}
				\right)^{\frac{k}{k-n}}
					\quad\text{for $r\ge0$}.
\end{equation}
Given an open set $\Omega \subset  \rn$, 
let $L(\widehat A,\frac{n}{k})(\Omega, \mathbb R^\ell)$ be the  Orlicz-Lorentz space  defined as in \eqref{sep35}. Namely, $L(\widehat A,\frac{n}{k})(\Omega, \mathbb R^\ell)$ is the rearrangement-invariant space associated with the function norm given by
\begin{equation}\label{E:29}
	\|f\|_{L(\widehat A,\frac{n}{k})(0,|\Omega|)}
		= \|r^{-\frac{k}{n}}f^{*}(r)\|_{L^{\widehat A}(0, {|\Omega|})}
\end{equation}
for $f \in \Mpl (0, |\Omega|)$.
The function $A$  always dominates $\widehat A$, and the two functions are equivalent if $A$ grows less than the critical power $t^{\frac ns}$ in the sense of the Matuszewska-Orlicz upper index (see e.g. \cite{cianchi_ibero} for details). In the latter case, $\widehat A$ can thus be simply replaced with $A$ in \eqref{E:29}.
\\
The space $L(\widehat A,\frac{n}{k})(\rn, \mathbb R^\ell)$ characterizes
the optimal  rearrangement-invariant  target space for embeddings of the Sobolev space $V^{\mathcal A_k}_\circ L^A(\rn, \mathbb{R}^\ell)$. This is the content of the next result, which follows from Theorem \ref{characterization}, via \cite[Inequalities (3.1) and (3.2)]{cianchi_ibero}. Its 
 statement takes a different form according to whether the assumption \eqref{convinf} or the complementary assumption 
\begin{equation}\label{divinf}
	\int^\infty\left(\frac{t}{A(t)}\right)^{\frac{k}{n-k}}\,dt = \infty
\end{equation}	
is in force.

\begin{theorem} [{\bf Orlicz-Sobolev inequalities with optimal rearrangement-invariant  target}]\label{sobolev-orliczlorentz-canc} 
  Assume that $n,m, \ell, k$ and $\mathcal A_k(D)$ be as in Theorem \ref{characterization}.
Let $A$ be a Young function fulfilling the condition \eqref{conv0}.
\\ (i) Assume that \eqref{divinf} holds.
Then, 
 $$V^{\mathcal A_k}_{{\circ}} L^A(\rn, \mathbb{R}^\ell) \to L(\widehat A,\tfrac{n}{k})(\rn, \R^\ell)$$
and
there exists a constant $c=c(A, \mathcal A_k(D))$ such that
\begin{equation}\label{bound-orliczlorentz1-canc}
\|u\|_{L(\widehat A,\frac{n}{k})(\rn, \R^\ell)} \leq c \|\mathcal A_k(D)u\|_{L^A(\rn, \mathbb R^m)}
\end{equation}
for every
$u \in V^{\mathcal A_k}_{{\circ}} L^A(\rn, \mathbb{R}^\ell)$.
\\ (ii) Assume that \eqref{convinf} holds. Then, 
$$V^{\mathcal A_k}_{{\circ}} L^A(\rn, \mathbb{R}^\ell) \to {(L^\infty \cap L(\widehat A,\tfrac{n}{k}))(\rn, \mathbb{R}^\ell)}$$
and
there exists a~constant    $c=c(A, \mathcal A_k(D))$ such that
\begin{equation}\label{bound-orliczlorentz2-canc}
\|u\|_{{(L^\infty \cap L(\widehat A,\frac{n}{k}))(\rn, \mathbb{R}^\ell)}} \leq c \|\mathcal A_k(D)u\|_{L^A(\rn, \mathbb R^m)}
\end{equation}
for every
$u \in V^{\mathcal A_k}_{{\circ}} L^A(\rn, \mathbb{R}^\ell)$.
\\ In both the inequalities \eqref{bound-orliczlorentz1-canc} and \eqref{bound-orliczlorentz2-canc} the target space is optimal (smallest possible) among all rearrangement-invariant spaces.
\end{theorem}

    \begin{remark}
        \label{sob-omega}
{\rm        A version of Theorem \ref{sobolev-orliczlorentz-canc} on open sets $\Omega$  in $\rn$ such that $|\Omega|<\infty$. In this case, the a priori condition \eqref{conv0} can be dropped, and 
          $A_{\frac nk}$ can be  defined as in \eqref{An} with $A$ modified, if necessary, near $0$ in such a way that this condition is fulfilled.
          Moreover, in Part (ii), the space $(L^\infty \cap L(\widehat A,\frac{n}{k}))(\Omega, \mathbb R^\ell)$ can just be replaced with  $L^\infty(\Omega, \mathbb R^\ell)$, as the two spaces agree under the current assumption on $\Omega$, up to equivalent norms.}
    \end{remark}

\begin{example}\label{ex orlicz-lorentz}{\rm
Consider a Young function $A$ as in \eqref{june3}--\eqref{june4}. Theorem \ref{sobolev-orliczlorentz-canc}
 tells us that, if 
\begin{equation}\label{june7}
\text{either $1\leq p< \frac nk$,  or $p=\frac nk$ and $r \leq \frac nk -1$,}
\end{equation}
then the inequality \eqref{bound-orliczlorentz1-canc} holds with
\begin{equation}\label{june8}
\widehat A(t)\,\,  \simeq\,\, \begin{cases} t^{p_0}(\log \frac 1t)^{r_0} & \quad \text{ if $1\leq {p_0}< \frac nk$ }
\\
t^{\frac nk}  (\log \frac 1t)^{r_0 - \frac nk} & \quad  \text{if ${p_0}=\frac nk$ and $r_0> \frac nk -1$}
\end{cases} \quad \text{near zero,}
\end{equation}
and
\begin{equation}\label{dec254}
\widehat A(t) \,\, \simeq \,\, \begin{cases} t^p (\log t)^r & \quad  \text{ if $1\leq p< \frac nk$ }
\\
t^{\frac nk}  (\log t)^{r - \frac nk} &  \quad \text{if  $p=\frac nk$ and $r < \frac nk -1$}
\\
t^{\frac nk}  (\log t)^{-1} (\log (\log t))^{-\frac nk}&  \quad \text{if  $p=\frac nk$ and $r = \frac nk -1$}
\end{cases} \quad \text{near infinity.}
\end{equation}
In particular, the   choice ${p_0}=p<\tfrac nk$ and $r_0 =r=0$ yields $\widehat A(t) =t^p$.
\\ From an application of \cite[Lemma 6.12, Chapter 4]{BennettSharpley}, one can deduce that, if $1\leq p=p_0<\frac nk$ and $r=0$, then 
\begin{align}\label{sep40}
    L(\widehat A,\tfrac{n}{k})(\rn, \mathbb R^\ell)= L^{\frac {np}{n-k p},p}(\log L)^{\frac r p}(\rn,  \mathbb R^\ell),
\end{align}
up to equivalent norms. Hence, the choice $p_0=p=1$ and $r_0=0$ yields the inequality
\begin{align}
    \label{exLZ}
    \|u\|_{L^{\frac n{n-\alpha},1,r}(\rn, \mathbb R^\ell)} \leq c \|\mathcal A_k(D)u\|_{L^1(\log L)^{r}(\rn, \rM)}
\end{align}
for every $u\in V^{\mathcal A_k}_{{\circ}}L^1(\log L)^{r}(\rn, \mathbb R^\ell)$.
Characterizations of the space $L(\widehat A,\tfrac{n}{k})(\rn,  \mathbb R^\ell)$, analogous to \eqref{sep40}, for ${p_0}=p=\tfrac nk$, in terms of 
 Lorentz-Zygmund or generalized  Lorentz-Zygmund spaces are also available -- see e.g. \cite[Example 1.2]{cianchi_ibero}.  Let us mention that the Sobolev counterpart 
of \eqref{exLZ} for the standard first order gradient was obtained in \cite[Theorem 1]{Greco-Moscariello}.
\\ In the light of Remark \ref{sob-omega},  inequalities parallel to \eqref{bound-orliczlorentz1-canc} and \eqref{bound-orliczlorentz2-canc} for Sobolev functions $u$ supported in open sets $\Omega$, with $|\Omega|<\infty$, hold even if the functions $A$ considered in this example do not satisfy the assumption \eqref{june4}.  The only relevant piece of information is indeed the behavior near $\infty$ of $A$ and $\widehat A$ described in \eqref{june3} and \eqref{dec254}.
}
\end{example}

\begin{example}\label{lorentzzygmund} {\rm We conclude with an application of Corollary \ref{sobolev-opt-canc-domain} to Lorentz-Zygmund spaces. For brevity, we limit ourselves to domain spaces whose  first index equals $1$, namely to spaces of the form $L^{(1,q,r)}(\Omega, \mathbb R^\ell)$, with $q\in [1, \infty)$.
 As explained in Section \ref{intro}, these are the most relevant in the present setting. 
In order to avoid introducing new classes of functions spaces, we also assume, for simplicity, that $r>-\frac 1q$. 
\\ Let $\Omega$ be an open set in $\rn$ such that $|\Omega|<\infty$. 

An application of Corollary \ref{sobolev-opt-canc-domain}, combined with a result of \cite{cavaliere-new} where information on the norm \eqref{may7-domain} is provided for the space 
$L^{(1,q,r)}(\Omega, \rM)$, tells us that 
%
%
    \begin{equation*}
 \|u\|_{L^{\frac n{n-k},q,r+1}(\Omega, \mathbb R^\ell)} \leq c \|\mathcal A_k(D)u\|_{L^{(1,q,r)}(\Omega, \rM)}
\end{equation*}
for some constant $c$ and   every $u \in V^{\mathcal A_k}_0L^{(1,q,r)}(\Omega, \mathbb R^\ell)$. 
}
\end{example}

\section{K-functionals for spaces of $k$-th order divergence-free vector fields} \label{sec:kfuncthigher}

Given $n\geq 2$ and $k \in\N$, the $k$-th order divergence operator acting on functions on $\rn$ which take values in $\R^{n^k}$ will be denoted by 
 ${\operatorname*{div}_k}$.  For $F \in C^\infty(\mathbb{R}^n,\mathbb{R}^{n^k})$, it is defined as
\begin{align}\label{smooth_div_k}
{\operatorname*{div_k}} F = \sum_{\beta \in \mathbb N^n, |\beta|=k} \partial^\beta F_\beta.
\end{align} 
 For $F\in L^1_{\rm loc}(\rn, \R^{n^k})$,  the  equality $\operatorname*{div}_kF=0$ is understood in the sense of distributions, namely:
\begin{align}\label{div_free_higherorder}
\int_{\mathbb{R}^n} F \cdot \nabla^k \varphi\;dx =0
\end{align}
for every $\varphi \in C^\infty_c(\rn)$.  \\
The   definitions above differ slightly from the convention of Van Schaftingen in \cite{VS3}, a point we now clarify.  The subspace of symmetric functions on $\rn$ with values in $\mathbb{R}^{n^k}$ has dimension
\begin{align}
    \label{N}
    N= {n+k-1 \choose k}.
\end{align}
For smooth functions on $\R^n$ with values in $\R^N$, the formula \eqref{smooth_div_k} is Van Schaftingen's definition of ${\operatorname*{div}_k}$.    With an abuse of notation we utilize this symbol to denote both of these differentiations. 
To mediate between the two, one identifies
a function $F$ with values in $\R^N$ and a tensor $\overline{F}$ with values in $\R^{n^k}$ that has $N$ non-zero components:
\begin{align*}
{\operatorname*{div_k}} F = {\operatorname*{div_k}}  \overline{F}.
\end{align*}
Alternatively, if one prefers to lift $F$ to a symmetric tensor, one needs to introduce combinatorial constants in the formula above.
\\
We write $X_{\operatorname*{div}_k}(\rn, \om)$ for the subspace of functions in a rearrangement-invariant space $X(\rn, \om)$ which satisfy $\operatorname*{div}_kF=0$ in the sense of distributions.

\medskip
\noindent
This section is devoted to the computation of the $K$-functional for the couple ($L^1_{\operatorname*{div}_k}(\rn, \rn),L^{p,q}_{\operatorname*{div}_k}(\rn, \rn)$). This is the content of the following result.


\begin{theorem}[{\bf $K$-functional for $k$-th order divergence--free vector fields}]\label{lemma2-higherorder}
Let $p \in (1,\infty)$, $q\in [1, \infty]$, $ k \in \mathbb{N}$, and let $N$ be defined by \eqref{N}.  Then,
\begin{align}\label{sep1higher}
K(F,t,L^1_{\operatorname*{div}_k}(\rn, \om),L^{p,q}_{\operatorname*{div}_k}(\rn, \om)) & \approx K(F,t,L^1(\rn, \om),L^{p,q}(\rn, \om))
\\ \nonumber & \approx
\int_0^{t^{p'}}F^*(s)\, ds + t \bigg(\int_{t^{p'}}^\infty s^{-1+\frac qp} F^*(s)\, ds\bigg)^{\frac 1q}  \qquad \text{for $t>0$,}
\end{align}
for every $F \in L^1 (\rn, \om) +L^{p,q}(\rn, \om)$ such that $\operatorname*{div}_kF=0$, with equivalence constants depending on $n, k, p, q$. 
\end{theorem}

\begin{remark}
    \label{rem-sep}
      {\rm Let us stress that 
    equation \eqref{sep1higher} does hold for every
$F \in L^1 (\rn, \mathbb R^N) +L^{p,q}(\rn, \mathbb R^N)$ such that $\operatorname*{div}_kF=0$, namely for $F\in  (L^1 (\rn, \mathbb R^N) +L^{p,q}(\rn, \mathbb R^N))_{\operatorname*{div}_k}$, and not just for $F\in  L^1_{\operatorname*{div}_k}(\rn, \mathbb R^N)+L^{p,q}_{\operatorname*{div}_k}(\rn, \mathbb R^N)$, as would be sufficient according to the definition of $K(F,t,L^1_{\operatorname*{div}_k}(\rn, \mathbb R^N),L^{p,q}_{\operatorname*{div}_k}(\rn, \mathbb R^N))$. In particular, this encodes the piece of information  that any function $F\in  (L^1 (\rn, \mathbb R^N) +L^{p,q}(\rn, \mathbb R^N))_{\operatorname*{div}_k}$ admits an essentially optimal decomposition of the form $F=F_1+F_{p,q}$ with $F_1 \in  L^1_{\operatorname*{div}_k}(\rn, \mathbb R^N)$ and $F_{p,q} \in L^{p,q}_{\operatorname*{div}_k}(\rn, \mathbb R^N)$.
This   is crucial in view of our applications of Theorem \ref{lemma2-higherorder}
.}
\end{remark}



The proof of Theorem \ref{lemma2-higherorder} is reminiscent of that of \cite[Lemma 2.4 on p.~170]{Bourgain}, of which our result is an extension. It 
makes critical use of the invariance of vector-valued functions composed with the $k$-th  order Helmholz projection operator under the constraint $\operatorname*{div}_kF=0$. This is the content of 
Lemma \ref{lemma1-higherorder} below.
 The relevant operator  is formally defined as 
\begin{align}
    \label{aug17prime} \mathcal H_kF = (-1)^{k-1}\nabla^k \operatorname*{div_k} (-\Delta)^{-k} F
\end{align}
 for   $F \in L^1(\rn, \om) +L^{p,q}(\rn, \om)$, where  
 $F$ is canonically identified with a symmetric tensor with $n^k$ components.  
 \\
Now, the Fourier calculus yields
\begin{align*}
\mathcal H_k \Phi = \left( -\frac{\xi^{\otimes k}}{|\xi|^k} \left(\frac{\xi^{\otimes k}}{|\xi|^k}\cdot \widehat{\Phi}(\xi)\right) \right)\widecheck{\phantom{x}}
\end{align*}
 if $\Phi \in C^\infty_c(\mathbb{R}^n,   \om)$. 
Since
\begin{align*}
\left|\frac{\partial^\beta}{\partial \xi^\beta} \frac{\xi^{\otimes k}}{|\xi|^k} \frac{\xi^{\otimes k}}{|\xi|^k} \right| \lesssim |\xi|^{-|\beta|} \qquad \text{for $\xi \neq 0$,}
\end{align*}
for every multi-index $\beta \in \mathbb{N}_0^n$,
 the operator $\mathcal H_k$ is bounded on $L^p(\mathbb{R}^n,\om)$ for $1<p<\infty$ and therefore also on $L^{p,q}(\mathbb{R}^n,  \om)$ for $1<p<\infty$ and $1\leq q \leq \infty$. 
 Indeed,
 one can invoke Mihlin's multiplier theorem \cite[Theorem 6.2.7 on p.~446]{grafakos} to prove   boundedness in $L^p$, while an interpolation argument \cite[Theorem 1.4.19 on p.~61]{grafakos} yields $L^{p,q}$ boundedness.
 \\
The kernel $\kappa : \rn \setminus \{0\} \to \R$ of the operator $\mathcal H_k$ fulfills the so called H\"ormander condition:
\begin{align}
    \label{hormander}
    \sup_{x\neq 0}\int_{\{|y|\geq 2|x|\}} |\kappa (y-x)- \kappa (y)|\, dy<\infty,
\end{align}
 see \cite[Proof of Theorem 6.2.7]{grafakos}.
\\ Notice that  
\begin{align}
    \label{aug18-higher}
    \mathcal H_k \,\nabla^k \varphi = -\nabla^k \varphi
\end{align}
for every $\varphi \in C^\infty_c(\mathbb{R}^n)$.

\begin{lemma}\label{lemma1-higherorder}
Let  $p\in (1,\infty)$, $q\in [1, \infty]$, $ k \in \mathbb{N}$, and let $N$ be defined by \eqref{N}.   Define the operator $P_k$ as
\begin{align*}
P_kF = F+\mathcal H_k F
\end{align*}
for  $F \in L^1(\rn, \om) +L^{p,q}(\rn, \om)$. 
\\ (i) If $\operatorname*{div}_kF=0$ in the sense of distributions, then
\begin{align}\label{aug20-higher}
P_kF=F.
\end{align}
 (ii) If $F, P_kF \in L^1(\rn, \om)$ or $F, P_kF \in L^{p,q}(\rn, \om)$, then  
\begin{align}\label{aug21-higher}
\operatorname*{div_k} P_kF = 0
\end{align}
in the sense of distributions.
\end{lemma}

\begin{proof} 
Throughout this proof, the constants in the relations $\lq\lq \lesssim"$ and $\lq\lq \approx"$ only depend on $n, k, p, q$.
\\
Part (i).
Let $F \in L^1(\rn, \om) + L^{p,q}(\rn, \om)$ be such that $\operatorname*{div}_kF=0$
 in the sense of distributions, i.e. it satisfies \eqref{div_free_higherorder}  for all $\varphi \in C^\infty_c(\mathbb{R}^n)$.   
We begin by showing that
\begin{align}\label{higher-order-claim}
\int_{\rn} F \cdot \nabla^k\operatorname*{div_k} (-\Delta)^{-k} \Phi\;dx = 0
 \end{align}
 for every $\Phi \in C^\infty_c(\mathbb{R}^n,\om)$, despite the fact that the test function $\varphi =\operatorname*{div}_k (-\Delta)^{-k} \Phi$ is not a priori admissible  in \eqref{div_free_higherorder}  for an arbitrary $\Phi \in C^\infty_c(\mathbb{R}^n,\om)$.
One has that
\begin{align}\label{phi-higherorder}
|\nabla^j \varphi(x)| \lesssim \frac{|\log (2+ |x|) |}{(1+|x|)^{n-k+j}} \quad \text{for $x\in \rn$,}
\end{align}
for $j=0\ldots k$, as follows  e.g., from  \cite[Theorem 2.1 on p.~209]{Mizuta}. 
We claim that 
 equation \eqref{higher-order-claim} will follow if we prove that
\begin{align}
    \label{dec24-12}
    \int_{\mathbb{R}^n} F \cdot \nabla^k ((\varphi\ast \rho_h)(1-\eta))\;dx=0,
\end{align}
where
$\{\rho_h\}$ is a standard mollifier supported in the ball $B_{1/h}$ for $h>0$,  and $\eta \in C^\infty_c(\mathbb{R}^n)$ is a cut-off function such that $\eta =1$ on $B_R$, $\eta =0$ outside $B(2R)$ and $|\nabla^j\eta|\lesssim 1/R^j$, $j=0,1,\dots,k$ for some $R>0$. Indeed, 
\begin{align}\label{aug15}
\int_{\mathbb{R}^n} F \cdot \nabla^k \varphi\;dx &= 
\lim_{h \to \infty} \bigg(\int_{\mathbb{R}^n} F \cdot \nabla^k ((\varphi\ast \rho_h) \eta)\;dx +\int_{\mathbb{R}^n} F \cdot \nabla^k ((\varphi\ast \rho_h)(1-\eta))\;dx\bigg) \\
\nonumber &= \lim_{h \to \infty} \int_{\mathbb{R}^n} F \cdot \nabla^k ((\varphi\ast \rho_h)(1-\eta))\;dx,
\end{align}
where the first equality holds thanks to the dominated convergence theorem, since $|\nabla^k \varphi| \in L^\infty(\rn)$, and the second one by \eqref{div_free_higherorder}.
\\ To prove \eqref{dec24-12}, observe that, if $F=F_1+F_{p,q}$ with $F_1\in L^1(\mathbb{R}^n,\om)$ and $F_{p,q}\in L^{p,q}(\mathbb{R}^n,\om)$, then
\begin{align}\label{aug8-higherorder}
\left|\int_{\mathbb{R}^n} F \cdot \nabla^k ((\varphi\ast \rho_h)(1-\eta))\;dx\right| &\lesssim  \int_{\rn \setminus B_R} (|F_1| +|F_{p,q}|)(|\nabla^k \varphi|* \rho_h)  \;dx \\ \nonumber
&\quad+ \sum_{j=0}^{k-1} \frac 1{R^{k-j}} \int_{B_{2R} \setminus B_R} (|F_1| +|F_{p,q}|)|(|\nabla^j \varphi| *  \rho_h)\;dx.
\end{align}
By the 
the inequality \eqref{phi-higherorder}, one has that 
\begin{align*}
\sum_{j=0}^{k-1} \frac 1{R^{k-j}}  \||\nabla^j \varphi| *  \rho_h  \|_{L^\infty(B_{2R} \setminus B_R)} + \||\nabla^k \varphi| *  \rho_h \|_{L^\infty(\mathbb{R}^n \setminus B_R)} \lesssim \frac{|\log R|}{R^n}.
\end{align*}
Since the latter quotient tends to zero as 
 $R \to \infty$, the integrals depending on $F_1$ on the right-hand side of \eqref{aug8-higherorder}  also tend to zero as 
 $R \to \infty$.
 \\ It remains to show that the same happens for the integrals 
depending on $F_{p,q}$. 
As a preliminary step, fix any measurable set $E\subset \rn$ and observe that, since ${\rm \supp} \rho_h \subset B_1$ for $h\in \N$, an application of Fubini's theorem and the inequality 
$$\chi_{B_1}(y)\chi_{B_{2R}\setminus B_R}(x) \leq \chi_{B_{2R+1}\setminus B_{R-1}}(x-y) \quad \text{for $x,y \in \rn$}$$
imply that
\begin{align}
    \label{aug13}
    \int_E (|\nabla^j \varphi|\ast\rho_h)\chi_{\rn \setminus B_R}\, dx \leq \int_E (|\nabla^j \varphi|\chi_{\rn \setminus B_{R-1}})\ast \rho_h\, dx
\end{align}
for $j=0, \dots k$.
Hence, via equation \eqref{aug11} and the inequality \eqref{nov1}, one deduces that
\begin{align}
    \label{aug14}
    (|\nabla^j \varphi|\ast\rho_h)\chi_{\rn \setminus B_R})^{**} (s) \leq (|\nabla^j \varphi|\chi_{\rn \setminus B_{R-1}})^{**}(s) \quad \text{for $s>0$.}
\end{align}
  Analogously to \eqref{aug14}, one has that
\begin{align}
    \label{aug14bis}
    (|\nabla^j \varphi|\ast\rho_h)\chi_{B_{2R} \setminus B_R})^{**} (s) \leq (|\nabla^j \varphi|\chi_{B_{3R} \setminus B_{R/2}})^{**}(s) \quad \text{for $s>0$,}
\end{align}
provided that $R>2$. 
Owing to the  inequalities \eqref{aug14} and \eqref{aug14bis}, and the property \eqref{hardy},
\begin{align}\label{oct1}
    \sum_{j=0}^{k-1} \|(|\nabla^j \varphi| &*  \rho_h ) \chi_{B_{2R} \setminus B_R}\|_{L^{p',q'}(\rn)} +\|(|\nabla^k \varphi| *  \rho_h ) \chi_{\rn \setminus B_R}\|_{L^{p',q'}(\rn)}\\
\nonumber&\leq \sum_{j=0}^{k-1} \||\nabla^j \varphi|  \chi_{B_{3R} \setminus B_{R/2}}\|_{L^{p',q'}(\rn)} +\||\nabla^k \varphi|   \chi_{\rn \setminus B_{R/2}}\|_{L^{p',q'}(\rn)}
 \end{align}
 for $R>2$.  
 Assume first that $j<k$. 
The inequalities \eqref{phi-higherorder} and the property (P2) for Lorentz norms imply that
 \begin{align}\label{oct1bis}
     \frac {1}{R^{k-j}} &\||\nabla^j \varphi|  \chi_{B_{3R} \setminus B_{R/2}}\|_{L^{p',q'}(\rn)} \lesssim  \frac {1}{R^{k-j}} \bigg\|\frac{|\log({2+}|x|)|}{(1+|x|^{n-k+j})}\chi_{B_{3R} \setminus B_{R/2}}\bigg\|_{L^{p',q'}(\rn)}\\
\nonumber&\lesssim  \frac {\log(R)}{R^{k-j}} \bigg\|\frac{1}{(1+|x|^{n-k+j})}\chi_{B_{3R} \setminus B_{R/2}}\bigg\|_{L^{p',q'}(\rn)}
     \\ \nonumber & \lesssim
     \frac {\log R}{R^{k-j}}
  \left(\int_0^\infty \left(t\left|\left\{ x \in B_{3R} \setminus B_{R/2} : \frac{1}{(1+|x|)^{n-k+j}} >t\right\}\right|^{1/p'}\right)^{q'}\;\frac{dt}{t}\right)^{1/q'}
 \\ & \nonumber \leq
     \frac {\log R}{R^{k-j}}
  \left(\int_0^{cR^{k-n-j}} \left(t\left|\left\{ x \in B_{3R} \setminus B_{R/2} : \frac{1}{(1+|x|)^{n-k+j}} >t\right\}\right|^{1/p'}\right)^{q'}\;\frac{dt}{t}\right)^{1/q'}
\\ &  \nonumber\leq
     \frac {\log R}{R^{k-j}}
  \left(\int_0^{cR^{k-n-j}} \left(t\left|\left\{ x \in B_{3R} \setminus B_{R/2} \right\}\right|^{1/p'}\right)^{q'}\;\frac{dt}{t}\right)^{1/q'}
 \\ & \nonumber \lesssim
     \frac { R^{n/p'} \log R}{R^{k-j}}
  \left(\int_0^{cR^{k-n-j}} t^{q'-1}\;dt\right)^{1/q'} 
  \\ \nonumber & \approx R^{-n/p} \log R
\end{align}
for some constant $c$ and for $R>2$.
\\ Consider next that case when $j=k$.
For every $\varepsilon >0$ there exists a positive constant $c=c(\varepsilon)$  such that $\log(\textcolor{red}{2+}|x|)\leq\,c(\varepsilon)(1+|x|)^{\varepsilon}$. 
If $\varepsilon$ is chosen so small that $1-\tfrac{n}{p'(n-\varepsilon)}>0$, then, by \eqref{phi-higherorder},
\begin{align}\label{oct2}
   &\||\nabla^k \varphi|  \chi_{\R^n \setminus B_{R/2}}\|_{L^{p',q'}(\rn)}\lesssim  \bigg\|\frac{1}{(1+|x|)^{n-\varepsilon}}\chi_{ \R^n \setminus B_{R/2}}\bigg\|_{L^{p',q'}(\rn)}\\
\nonumber & \lesssim
  \left(\int_0^\infty \left(t\left|\left\{ x \in \R^n \setminus B_{R/2} : \frac{1}{(1+|x|)^{n-\varepsilon}} >t\right\}\right|^{1/p'}\right)^{q'}\;\frac{dt}{t}\right)^{1/q'}
 \\ & \nonumber \leq
  \left(\int_0^{cR^{\varepsilon-n}} \left(t\left|\left\{ x \in \R^n \setminus B_{R/2} : \frac{1}{(1+|x|)^{n-\varepsilon}} >t\right\}\right|^{1/p'}\right)^{q'}\;\frac{dt}{t}\right)^{1/q'}\\
 & \nonumber \lesssim
\left(\int_0^{cR^{\varepsilon-n}}t^{-\frac{nq'}{p'(n-\varepsilon)}} t^{q'-1}\;dt\right)^{1/q'} 
  \\ \nonumber & \approx R^{-(n-\varepsilon)\gamma},
\end{align}
where we have set $\gamma=1-\tfrac{n}{p'(n-\varepsilon)}$.
\\ From equations \eqref{oct1}--\eqref{oct2} we deduce that also the integrals depending on $F_{p,q}$ on the right-hand side of the inequality \eqref{aug8-higherorder} tend to zero as $R\to \infty$. Equation \eqref{higher-order-claim} is thus established.
{
\\ Hence, by equation \eqref{aug17prime}, we have
\begin{align*}
\int_{\mathbb{R}^n} F \cdot \Phi \;dx =\int_{\mathbb{R}^n} F \cdot (\Phi+\mathcal H_k\Phi) \;dx
=\int_{\mathbb{R}^n} F \cdot P_k\Phi \;dx =\int_{\mathbb{R}^n} P_kF \cdot \Phi \;dx,
\end{align*}
where the last equality follows via an integration by parts and Fubini's theorem. This proves equation \eqref{aug20-higher}.
\\ Part (ii). We have to show that, 
if $F,P_k{F} \in L^1(\rn, \rn)$ or $F,P_kF \in L^{p,q}(\rn, \rn)$, then
\begin{align}\label{aug23}
\int_{\mathbb{R}^n} P_kF \cdot \nabla^k \varphi \;dx =0
\end{align}
for every $\varphi \in C^\infty_c(\mathbb{R}^n)$.  One has that
\begin{align}\label{aug22}
\int_{\mathbb{R}^n} P_kF \cdot \nabla^k \varphi \;dx &= \int_{\mathbb{R}^n} F \cdot \nabla^k \varphi \;dx+ \int_{\mathbb{R}^n} \mathcal H_k F \cdot \nabla^k \varphi \;dx.
\end{align}
The assumptions on $F,P_kF$ imply $ \mathcal H_k F \in L^1(\rn, \rn)$ or $ \mathcal H_k F\in L^{p,q}(\rn, \rn)$. Either of these  integrability properties  suffices to ensure that 
\begin{align*}
    \int_{\mathbb{R}^n} \mathcal H_k F \cdot \nabla^k \varphi \;dx &= \int_{\mathbb{R}^n}   F \cdot \mathcal H_k\, \nabla^k \varphi \;dx.
\end{align*}
Combining the latter equality with 
 equations \eqref{aug18-higher} and \eqref{aug22} yields \eqref{aug23}.
} 
\end{proof}

{
\begin{proof}
    [Proof of Theorem \ref{lemma2-higherorder}]
All functions appearing throughout this proof map $\rn$ into $\mathbb R^N$. Since there will be no ambiguity, we drop the notation of the domain and the target in the function spaces. The  constants in the relations $\lq\lq \lesssim"$ and $\lq\lq \approx"$ only depend on $n, k, p, q$.
\\
    Recall that, according to the definition \eqref{Kfunct}, 
    \begin{align*}
        K(F,t,L^1,L^{p,q}) &= \inf \{ \|F_1\|_{L^1}+ t\|F_{p,q}\|_{L^{p,q} } :F=F_1+F_{p,q}\}\quad \text{for $t>0$}
        \end{align*}
        for $F\in L^1 +L^{p,q}$,
        and 
           \begin{align*}
         K(F,t,L^1_{\operatorname*{div}_k} ,L^{p,q}_{\operatorname*{div}_k} )  = \inf \{ \|F_1\|_{L^1 }+ t\|F_{p,q}\|_{L^{p,q} } : F=F_1+F_{p,q}, {\rm div}_k F_1={\rm div}_kF_{p,q} =0 \}\quad \text{for $t>0$}
    \end{align*}
    for $F\in L^1_{\operatorname*{div}_k} +L^{p,q}_{\operatorname*{div}_k}$. 
    \\ Fix $t>0$ and $F \in L^1   +L^{p,q}$ such that $\operatorname*{div}_kF=0$. Thanks to a larger class of admissible decompositions in the computation of the $K$-functional on its left-hand side, 
 the inequality
    \begin{align*}
       K(F,t,L^1 ,L^{p,q} ) \leq K(F,t,L^1_{\operatorname*{div}_k},L^{p,q}_{\operatorname*{div}_k})
    \end{align*}
    is trivial.  To establish the first equivalence in \eqref{sep1higher}, it therefore remains to show that, up to a multiplicative constant, the reverse inequality also holds, namely:
    \begin{align}\label{sep5}
    K(F,t,L^1_{\operatorname*{div}_k},L^{p,q}_{\operatorname*{div}_k})  \lesssim K(F,t,L^1,L^{p,q}).
    \end{align}
By scaling, we may assume
\begin{align}\label{sep7}
    K(F,t,L^1,L^{p,q}) =1,
    \end{align}
and then we must show that
\begin{align}\label{sep6}
   K(F,t,L^1_{\operatorname*{div}_k},L^{p,q}_{\operatorname*{div}_k}) \lesssim 1.
    \end{align}

In order to prove the inequality \eqref{sep6} under \eqref{sep7}, 
consider any decomposition  $F=F_1+F_{p,q}$  of $F$   such that
\begin{align*}
   \|F_1\|_{L^1}+ t\|F_{p,q}\|_{L^{p,q}} \leq 2.
\end{align*}
The Calder\'on-Zygmund decomposition \cite[Theorem 5.3.1 on p.~355]{grafakos} of $F_1$, with $\lambda = t^{-p'}$, 
yields $$F_1=H+K$$ for some functions $H, K \in L^1$ such that:
\begin{align*}
|H| &\leq \lambda ,\\
\|H\|_{L^1} &\leq \|F_1\|_{L^1} \leq 2,
\end{align*}
and $$K=\sum_i K_i$$ for suitable functions $K_i\in L^1$ satisfying, for balls $B_i\subset\R^n$,
\begin{align*}
\operatorname*{supp} K_i &\subset B_i,
\\ \int_{B_i} K_i \,dx&= 0,\\
\sum_i |B_i| &\lesssim \|F_1\|_{L^1}
\lambda^{-1} \leq 2\lambda^{-1},\\
\sum_i \|K_i\|_{L^1} &\lesssim \|F_1\|_{L^1} \leq 2.
\end{align*}
By Lemma  \ref{lemma1-higherorder}, Part (i), we have that $F=P_kF$. Therefore,
\begin{align}\label{aug1}
F= P_kF = P_k(H+F_{p,q})+P_kK.
\end{align}
If we 
 show that $P_kK \in L^1, P_k
(H+F_{p,q})\in L^{p,q}$, and 
\begin{align}\label{final_claim}
\|P_kK\|_{L^1} + t\|P_k(H+F_{p,q})\|_{L^{p,q}} \lesssim 1,
\end{align}
then we can conclude that 
\eqref{aug1} is an admissible decomposition for the $K$-functional for the couple $(L^1_{\operatorname*{div}_k} ,L^{p,q}_{\operatorname*{div}_k})$, since, by  Lemma  \ref{lemma1-higherorder}, Part (ii),
$${{\rm div}_k}P_kK=0 \quad \text{and} \quad {\rm div}_k P_k
(H+F_{p,q})=0.$$
Hence \eqref{sep6} will follow via \eqref{final_claim}.
\\ To complete the proof, it thus only remains
to prove the bound \eqref{final_claim}.
%
Concerning the second addend on the left-hand side of  \eqref{final_claim}, by the boundedness of $\mathcal H_k$, and hence of $P_k$, on 
 $L^{p,q}$
%
 and the inequality \eqref{aug3},
one has
\begin{align}\label{aug2}
\|P_k(H+F_{p,q})\|_{L^{p,q}} &  \lesssim \|H\|_{L^{p,q}}+  \|F_{p,q}\|_{L^{p,q}}\\ \nonumber
&\lesssim \|H\|_{L^1}^{1/p}\|H\|_{L^\infty}^{1/p'}+\|F_{p,q}\|_{L^{p,q}}\\ \nonumber
&\lesssim \lambda^{1/p'} +t^{-1}\\ \nonumber
&\approx t^{-1}.
\end{align}
Turning our attention to the bound for the first addend on the right-hand side of \eqref{final_claim}, we define $\Omega = \cup_i B_i^*$, where $B_i^{*}$ is the ball with the same center as $B_i$ with twice the radius.  Then,
\begin{align*}
\|P_kK\|_{L^1} = \|P_kK \chi_\Omega\|_{L^1} + \|P_kK \chi_{\Omega^c}\|_{L^1}.
\end{align*}
Inasmuch as $\mathrm{supp} \,K_i\subset B_i$ and
 the kernel of the operator $\mathcal H_k$  satisfies H\"ormander's condition \eqref{hormander}, by \cite[Inequality (2.13)]{Bourgain} 
$$ \sum_i \|P_kK_i \chi_{(B_i^*)^c}\|_{L^1}  =\sum_i \|\mathcal H_k K_i \chi_{(B_i^*)^c}\|_{L^1}  \lesssim \sum_i \|K_i\|_{L^1}.$$
Hence,
\begin{align}\label{aug7}
\|P_kK \chi_{\Omega^c}\|_{L^1} &\leq \sum_i \|P_kK_i \chi_{(B_i^*)^c}\|_{L^1}  \lesssim \sum_i \|K_i\|_{L^1} \lesssim \|F\|_{L^1}.
\end{align}
Since $F=H+K+F_{p,q}$, from equation \eqref{aug1} we deduce that
\begin{align*}
P_kK = K+F_{p,q}+H - P_k(F_{p,q}+H).
\end{align*}
Therefore,   the boundedness of $P_k$ on $L^{p,q}(\rn)$, the  H\"older type  inequality in the Lorentz spaces \eqref{holderlor}, and the fact that $\|\chi_{\Omega}\|_{L^{p',q'}}$ is independent of $q$ yield
\begin{align*}
\|P_kK\chi_\Omega\|_{L^1} &\leq \|K\|_{L^1} + |\Omega|^{1/p'} \|F_{p,q}+H\|_{L^{p,q}}\\
&\lesssim \|F_1\|_{L^1} + |\Omega|^{1/p'} t^{-1} \lesssim 1 + \lambda^{-1/p'}t^{-1}  \approx 1
\end{align*}
by our choice of $\lambda$.  This completes the proof of the bound \eqref{final_claim} and thus also the proof of the first equivalence in \eqref{sep1higher}.
\\ The second equivalence holds thanks to Holmsted's formulas \cite[Theorem 4.1]{Hol}.
\end{proof}}

\section{A pivotal rearrangement inequality}\label{S:rearrangement}

A fundamental  tool in our proof of Theorem   \ref{characterization} is an estimate, in rearrangement form, for Riesz potentials,  composed with Calder\'on-Zygmund operators and linear operators on finite dimensional spaces, of $k$-th order divergence free vector fields. It relies upon the results of Section  \ref{sec:kfuncthigher} and 
 is the content of Theorem \ref{K-CZ} below,    the main technical advance of this contribution. 
\\ Recall that 
the Riesz potential operator 
$I_\alpha$, with $\alpha \in (0, n)$, is   defined on locally integrable functions $F: \rn \to \mathbb R^m$ as 
\begin{align*}
I_\alpha F(x) = \frac{1}{\gamma(\alpha)}\int_{\mathbb{R}^n} \frac{F(y)}{|x-y|^{n-\alpha}}\;dy \quad \text{for $x\in \rn$,}
\end{align*}
whenever the integral on the right-hand side is finite. Here, $n, m \in \N$, and 
 $\gamma(\alpha)$ is 
 a suitable normalization constant.

  {\color{black}
 \begin{theorem}\label{K-CZ}
  Let $k \in \N$, $n, l, \ell \geq 2$, and let $N$ be as in \eqref{N}. Let $\alpha \in (0,n)$. Let   $\{T^\beta: \beta \in \N^n,|\beta|=k\}$ be a family of linear operators such that 
  $T^\beta: L^{p}(\rn,\R^l) \to  L^{p}(\rn,\mathbb{R}^\ell)$ for every $p\in (1,\infty)$.   Assume that 
 \begin{equation}
     \label{commute}
     I_\alpha T^\beta F^\beta = T^\beta I_\alpha F^\beta
 \end{equation}
 for every $F\in L^1 (\rn,\R^{N\times l})$, where $F=[F^\beta]$ with rows $(F^\beta)_i$ for $i = 1,\ldots,l$.  Set
 \begin{align*}
   G= \sum_{|\beta|=k} I_\alpha T^\beta F^\beta.
 \end{align*}
Then, there exists a positive constant $c=c(\alpha, \{T^\beta\})$ such that
\begin{align}\label{nov100}
\int_0^t  s^{-\frac{\alpha}{n}} G^*(s)\, ds
\leq c\int_0^{t}  s^{-\frac{\alpha}{n}} \int_{s}^{\infty} F^*(r)r^{-1+\alpha/n}\,drds  \quad\text{for $t>0$,}
\end{align}

for every 
$F \in  L^1(\rn, \mathbb{R}^{N\times l})+ L^{\frac n\alpha, 1}(\rn, \mathbb{R}^{N\times l})$ such that ${\rm div_k}(F^\beta)_i=0$ for $i =1, \ldots,l$.

 \end{theorem}

\begin{remark}
    \label{rem-Riesz} {\rm Choosing the identity operators as $T^\beta$ in Theorem \ref{K-CZ} results in the estimate \eqref{nov100} for $I_\alpha F$.}
\end{remark}

\begin{remark}
    \label{rem-Kfunct}
    {\rm As will be clear from the proof, the 
    inequality \eqref{nov100} is equivalent to the $K$-functional inequality
\begin{equation}\label{dec32second}
K\big(G, t;   L^{\frac{n}{n-\alpha}, 1}(\rn, \R^l),  L^\infty(\rn, \R^l)\big)
\leq c K\big(F, t/c;   L^1_{\operatorname*{div_k}}(\rn, \R^{N\times l}),   L^{\frac n\alpha, 1}_{\operatorname*{div_k}}(\rn, \R^{N\times l})\big)
\end{equation}
for every $F \in  L^1(\rn, \mathbb{R}^{N\times l})+ L^{\frac n\alpha, 1}(\rn, \mathbb{R}^{N\times l})$ such that 
${\rm div_k}(F^\beta)_i=0$.}
\end{remark}

\begin{proof}[Proof of Theorem \ref{K-CZ}] 
Throughout this proof, the constants in the relations \lq\lq $\approx$"
 and \lq\lq $\lesssim$" depend only on $\frac n\alpha$ and $k$.
Observe that   $F \in  L^1 (\rn, \mathbb{R}^N)+ L^{\frac n\alpha, 1} (\rn, \mathbb{R}^N)$ if and only if 
$$\int_0^{t} F^*(s)\;ds + t^{1-\alpha/n} \int_{t}^\infty  s^{-1+\frac \alpha n} F^*(s)\;ds<\infty \quad \text{for $t>0$.}$$
This is a consequence  Holmsted's formulas -- see the second equivalence in \eqref{sep1higher}.
On the other hand, an application of Fubini's theorem tells us that 
\begin{align}
    \label{nov130}
     \int_0^{t} F^*(s)\;ds + t^{1-\alpha/n} \int_{t}^\infty  s^{-1+\frac \alpha n} F^*(s)\;ds  =   \frac {n-\alpha}n \int_0^{t} s^{-\alpha/n}   \int_{s}^\infty r^{-1+\alpha/n}F^*(r)\;drds
\end{align}
for $t>0$.
\\ 
    From \cite[Theorem 1]{Stolyarov} one has that
 \begin{equation}\label{dec30}
 I_\alpha : L^1_{\operatorname*{div_k}}(\rn, \R^{N\times l}) \to L^{\frac n{n-\alpha}, 1}(\rn, \R^{N\times l}),
 \end{equation}
  with norm depending on $n$ and $\alpha$.  
  Moreover,
  thanks to the boundedness of $T^\beta$ in $L^p$ for every $p>1$, a standard interpolation theorem for intermediate Lorentz spaces ensures that $T^\beta: L^{\frac {n}{n-\alpha}, 1}(\mathbb{R}^n,\mathbb{R}^l) \to L^{\frac {n}{n-\alpha}, 1}(\mathbb{R}^n,\mathbb{R}^\ell)$.
The commuting property  \eqref{commute} and this boundedness property of $T^\beta$, in combination with \eqref{dec30}, yield
 \begin{align}\label{nov110}
 \left\|\sum_{|\beta|=k} |I_\alpha T^\beta F^\beta| \right\|_{L^{\frac {n}{n-\alpha}, 1}(\rn)}&  =  \left\|\sum_{|\beta|=k} | T^\beta I_\alpha F^\beta| \right\|_{L^{\frac {n}{n-\alpha}, 1}(\rn)} \\ \nonumber 
 &\lesssim  \sum_{|\beta|=k} \left\| I_\alpha F^\beta \right\|_{L^{\frac {n}{n-\alpha}, 1}(\rn, \mathbb{R}^l)} \\ \nonumber
 &\lesssim \|I_\alpha F \|_{L^{\frac {n}{n-\alpha}, 1}(\rn, \R^{N\times l})}\\ \nonumber
 &\lesssim \|F \|_{L^{1}(\rn, \R^{N\times l})}.
\end{align}
for all $F \in  L^1_{\rm div_k}(\rn, \R^{N\times l})$. 
\\    The boundedness of $T^\beta$ in $L^p$ for every $p>1$ and the interpolation theorem for intermediate Lorentz spaces again ensure that $T^\beta: L^{\frac {n}{\alpha}, 1}(\rn,\R^\ell) \to L^{\frac {n}{\alpha}, 1}(\rn,  \R^\ell)$. This piece of information and  the boundedness of $I_\alpha : L^{\frac {n}{\alpha}, 1}(\rn) \to L^{\infty}(\rn)$  enable one to deduce that
\begin{align}\label{nov111}
 \left\|\sum_{|\beta|=k} |I_\alpha T^\beta F^\beta| \right\|_{L^{\infty}(\rn)} &\lesssim \sum_{|\beta|=k} \|T^\beta F^\beta \|_{L^{\frac {n}{\alpha}, 1}(\rn,  \R^\ell)} \\ \nonumber 
 &\lesssim  \sum_{|\beta|=k} \|F^\beta \|_{L^{\frac {n}{\alpha}, 1}(\rn, \mathbb{R}^l)} \\ \nonumber 
 &\lesssim  \| F \|_{L^{\frac {n}{\alpha}, 1}(\rn, \mathbb{R}^{N\times l})}
\end{align}
 for every $F \in  L^{\frac n\alpha, 1} (\rn, \mathbb{R}^{N\times l})$, and hence, a fortiori, for $F \in  L^{\frac n\alpha, 1}_{\rm div_k}(\rn, \mathbb{R}^{N\times l})$.
 \\
Let $F \in  L^1(\rn, \mathbb{R}^{N\times l})+ L^{\frac n\alpha, 1}(\rn, \mathbb{R}^{N\times l})$  be such that ${\rm div_k}F=0$ row-wise.   Here, for convenience of notation, we have suppressed the dependence in $i=1,\ldots,l$. By Theorem \ref{lemma2-higherorder}, such a function $F$ admits a decomposition $F=F_1 + F_{n/\alpha,1}$, with $F_1\in L^{1}_{\rm div_k}(\rn, \mathbb{R}^{N\times l})$ and $F_{n/\alpha,1}\in  L^{\frac n\alpha, 1}_{\rm div_k}(\rn, \mathbb{R}^{N\times l})$ fulfilling the estimate:
\begin{align}
    \label{nov134}
    \|F_1\|_{L^{1}(\rn, \mathbb{R}^{N\times l})}+ t \|F_{n/\alpha,1}\|_{L^{\frac n\alpha, 1}(\rn, \mathbb{R}^{N\times l})} \lesssim  \int_0^{t} F^*(s)\;ds + t^{1-\alpha/n} \int_{t}^\infty  s^{-1+\frac \alpha n} F^*(s)\;ds
\end{align}
for $t>0$.
\\ Equations \eqref{nov110} and \eqref{nov111} ensure that
$\sum_{|\beta|=k} I_\alpha T^\beta F^\beta_1 \in L^{\frac{n}{n-\alpha}, 1}(\rn,\R^l)$ and $\sum_{|\beta|=k}  I_\alpha T^\beta F^\beta_{n/\alpha,1} \in L^\infty(\rn,\R^l)$.
Therefore,
\begin{align}
   \nonumber
    K\big( G, t;   L^{\frac{n}{n-\alpha}, 1}(\rn,\R^l),  L^\infty(\rn, \R^l)\big) & \leq \left\|\sum_{|\beta|=k} I_\alpha T^\beta F^\beta_1 \right\|_{L^{\frac{n}{n-\alpha}, 1}(\rn, \R^l)} + t \left\|\sum_{|\beta|=k} I_\alpha T^\beta F_{n/\alpha,1}\right\|_{ L^\infty(\rn, \R^l)}
    \\ & \lesssim \|F_1\|_{L^{1}(\rn, \mathbb{R}^{N\times l})}+ t \|F_{n/\alpha}\|_{L^{\frac n\alpha, 1}(\rn, \mathbb{R}^{N\times l})}.
    \label{nov132}
\end{align}
Thanks to \cite[Corollary 2.3, Chapter 5]{BennettSharpley}, 
\begin{equation}\label{dec33}
K\big( G, t;   L^{\frac{n}{n-\alpha}, 1}(\rn,\R^l),  L^\infty(\rn, \R^l)\big) \approx \int_0^{t^{\frac n{n-\alpha}}} s^{-\frac{\alpha}{n}}G^*(s) \, ds \quad\text{for $t>0$,}
\end{equation} 
for $G\in L^{\frac n{n-\alpha}, 1}(\rn, \om)+  L^\infty(\rn, \om)$. 
\\ Combining equations \eqref{nov134}--\eqref{dec33}   with \eqref{nov130} yields the inequality \eqref{nov100}.
\end{proof}
 
 }

 \section{Proofs of the main results}\label{proof}

 The following 
 result from \cite[Proof of Theorem A]{KermanPick} (see also \cite[Proof of Theorem 4.1]{CPS_Frostman} for an alternative simpler argument) is needed in  our proof of Theorem \ref{characterization}. Thanks to a representation formula in terms of any elliptic canceling operator, it enables us to transfer the information contained in the inequalities \eqref{nov100} and \eqref{sobolev3} into the Sobolev inequality \eqref{characterization}.
 %
 %
 %

 \begin{theoremalph}\label{KermanPick} Let $n\in\mathbb N$ and $1 \leq k <n$. Let $\|\cdot\|_{X(0,\infty)}$ and $\|\cdot\|_{Y(0,\infty)}$ be  rearrangement-invariant function norms such that the inequality \eqref{sobolev3} holds. Suppose that the functions $f, g \in \mathcal M (0, \infty)$ are such that
\begin{align}\label{july21}
\int_0^t  s^{-\frac{k}{n}} g^*(s)\, ds
\leq c\int_0^{t}  s^{-\frac{k}{n}} \int_{s/c}^{\infty} f^*(r)r^{-1+k/n}\,drds  \quad\text{for $t>0$,}
\end{align}
for some positive constant   $c$.  Then
\begin{align}\label{july22}
\|g\|_{Y(0, \infty)} \leq c' \|f\|_{X(0, \infty)},
\end{align}
for a suitable constant $c'=c'(c, \frac nk)$. 
 \end{theoremalph}

The $k$-th order divergence operator is a special member of the class 
of co-canceling differential operators introduced in \cite[Definition 1.3]{VS3} and defined below. The curl is another classical instance of a first order co-canceling operator. Co-canceling operators come into play in the proof of Theorem \ref{characterization}.

\begin{definitionalph}[{\bf Co-canceling operator}]\label{co-canc}
{\rm Let $n, m \geq 2$ and $l \geq 1$. 
  A linear homogeneous $k$-th order constant coefficient differential operator 
  $
 \mathcal L(D)$ mapping $\mathbb{R}^m$-valued functions to $\mathbb{R}^l$-valued functions
is said to be co-canceling if there exist linear operators  $L_\beta: \rM \to \mathbb R^l$, with $\beta \in \mathbb N^n$, such that
\begin{align}\label{operator}
    \mathcal L(D)F = \sum_{\beta \in \mathbb N^n
,\, |\beta|=k}L_\beta(\partial ^\beta F)
\end{align}
for $F \in C^\infty(\mathbb{R}^n,\mathbb{R}^m)$, and 
\begin{align*}
\bigcap_{\xi\in\mathbb{R}^n \setminus \{0\}}\ker \mathcal L(\xi)=\{0\},
\end{align*}
where  $\mathcal L(\xi)$ denotes the symbol map of $\mathcal L(D)$ in terms of Fourier transforms.}
\end{definitionalph}

\begin{proof}[Proof of Theorem \ref{characterization}]
The implication $\lq\lq (i) \Rightarrow (ii)"$ follows from the pointwise inequality
\begin{align}
    \label{dec24-30}
    |\mathcal A_k(D)u| \lesssim |\nabla ^ku|,
\end{align}
which holds for every operator $\mathcal A_k$ and every function $u$, up to a multiplicative constant depending on $\mathcal A_k$.
\\ The implication $\lq\lq (ii) \Rightarrow (iii)"$ is established in  \cite{mihula}, where    results of \cite{KermanPick, cianchi-pick-slavikova} for functions defined on  domains with finite measure are generalized to Sobolev type spaces defined in $\rn$.
\\ The implication $\lq\lq (iii) \Rightarrow (i)"$ is the  novelty of this theorem. 
Denote by
$T_k$  the $(-n+k)$-positively homogeneous fractional integral operator given by
$$T_k=
((\mathcal A_k^*\mathcal A_k)^{-1}\mathcal A_k^*)\widecheck{\phantom{x}},$$
and by $T$  
 the Calder\'on-Zygmund operator 
 $$T = ((2\pi)^k|\cdot|^{k}(\mathcal A_k^*\mathcal A_k)^{-1}\mathcal A_k^*)\widecheck{\phantom{x}}.$$
 {\color{black}
We remark here that
 \begin{equation}\label{dec24-15}
 I_kT= TI_k=T_k
 \end{equation}
 on $L^1(\mathbb{R}^n,\mathbb{R}^m) +L^{n/k,1}(\mathbb{R}^n,\mathbb{R}^m)$.  Indeed, the preceding computation of the Fourier symbol shows the validity of the identity \eqref{dec24-15} when acting on $C^\infty_c(\mathbb{R}^n,\mathbb{R}^m)$, and as is usual for singular integral operators, the   equation \eqref{dec24-15} for general functions in  $L^1(\mathbb{R}^n,\mathbb{R}^m)$ and $L^{n/k,1}(\mathbb{R}^n,\mathbb{R}^m)$ follows by the density of $C^\infty_c(\mathbb{R}^n,\mathbb{R}^m)$ in these spaces.
 \\
%
Next, we claim the representation formula
\begin{align}\label{representation}
u= I_k T  \mathcal A_k(D)u = T_k \mathcal A_k(D)u
\end{align}
for $u\in V_\circ^{\mathcal A_k} (L^1 +L^{n/k,1})(\rn, \mathbb R^\ell)$.
\\
Any function $u\in V_\circ^{\mathcal A_k} (L^1 +L^{n/k,1})(\rn, \mathbb R^\ell)$
can be approximated, via standard mollification, by a sequence of functions in $C^\infty_\circ(\rn, \mathbb R^\ell)$.  To verify this assertion, we begin by observing that  
our assumption
\begin{align}\label{dec24-17}
\lim_{R \to \infty} \, R^j \dashint_{B_{2R} \setminus B_R} |\nabla^j u|\;dx = 0
\end{align}
for $j=0, \dots k-1$ implies
\begin{align}\label{dec24-18}
\lim_{R \to \infty}\, R^j \dashint_{B_{4R} \setminus B_{R/2}} |\nabla^j u|\;dx = 0.
\end{align}
This is an easy consequence of the identity
\begin{align*}
B_{4R} \setminus B_{R/2} = B_{R} \setminus B_{R/2} \cup B_{2R} \setminus B_R \cup B_{4R} \setminus B_{2R},
\end{align*}
  of \eqref{dec24-17}, and of \eqref{dec24-17} applied with $R$ replaced with $R/2$ and $2R$. 
  \\ Next, given a sequence $\{\rho_h\}$ of mollifiers such that $\supp \rho_h \subset B_1$ for $h \in \N$, one has that
\begin{align}\label{dec24-19}
    \dashint_{B_{2R} \setminus B_R} |\nabla^j (u*\rho_h)(x)|\;dx &= \frac{c}{R^n}\int_{B_{2R} \setminus B_R} \left| \int_{  B_1} \nabla^j u(x-y)\rho_h(y)\;dy \right|\;dx
    \\ \nonumber
    &\leq \frac{c}{R^n} \int_{B_1} \int_{B_{2R} \setminus B_R}| \nabla^j u(x-y)|  \;dx\rho_h(y)\;dy
    \\ \nonumber
    &  \leq \frac{c}{R^n}
\int_{\rn}\int_{\rn }
    | \nabla^j u(x-y)|\chi_{B_{2R+1} \setminus B_{R-1}}(x-y)\rho_h(y)\;dy\;dx   
     \\ \nonumber
    &  =\frac{c}{R^n} \int_{B_{2R+1} \setminus B_{R-1}}| \nabla^j u|  \;dz
    \\ \nonumber
    &\leq c' \    \dashint_{ B_{4R} \setminus B_{R/2}} |\nabla^j u|\;dz
\end{align}  
for some constants $c=c(n)$ and $c'=c'(n)$, and for $j=0, \cdots , k-1$, provided that $R>2$. Equations \eqref{dec24-18} and \eqref{dec24-19} imply that $u*\rho_h \in C^\infty_\circ(\rn, \mathbb R^\ell)$ for $h \in \N$.
  \\
Thus, since singular integral operators are defined by density via smooth functions, it suffices to prove the formula \eqref{representation} for $u\in C^\infty_\circ (\mathbb{R}^n;\mathbb{R}^\ell)$. 
 In the special case when $u \in C^\infty_c(\mathbb{R}^n;\mathbb{R}^\ell)$, the formula \eqref{representation} follows by taking the Fourier transform. Assume that $u\in C^\infty_\circ (\mathbb{R}^n;\mathbb{R}^l)$. Fix $x \in \rn$.
 Given $R>0$ let $\eta_R \in C^\infty_c(B_{{ 3R/2}}(x))$ be such that
\begin{align}\label{dec24-22}
 \nonumber \eta_R &= 1 \text{ on } B_R(x) \\ \nonumber
\nabla^j \eta_R  &= 0 \text{ on } B_R(x) \text{ for } j=1,\ldots, k\\
|\nabla^j \eta_R| &\leq c \frac{1}{R^j} \text{ for } j=0,\ldots, k,
\end{align}
  for some constant $c=c(n,j)$.
An application of the identity \eqref{representation} to $u\eta_R$ yields
\begin{align*}
u\eta_R = T_k\mathcal A_k(D)u + T_k\mathcal A_k(D)((\eta_R-1)u).
\end{align*}
Therefore, in order to establish  \eqref{representation} for $u$ it suffices to show that
\begin{align}\label
{dec24-26}
\lim_{R \to \infty} |T_k\mathcal A_k(D)((\eta_R-1)u)(x)| =0 \quad \text{for $x\in \rn$.}
\end{align}
To this end,    denote by $\mathcal T_k$ the kernel of the operator $T_k$, fix $x\in \rn$ and write
\begin{align*}
T_k\mathcal A_k(D)((\eta_R-1)u)(x) &= \int_{\mathbb{R}^n} \mathcal T_k(x-y)\cdot \mathcal A_k(D)((\eta_R-1)u)(y) \;dy\\
&= \int_{\mathbb{R}^n}  \mathcal T_k(x-y)\cdot \mathcal A_k(D)u(y)(\eta_R(y)-1) \;dy\\
&\quad +\; \sum_{j=0}^{k-1} \int_{\mathbb{R}^n}  \mathcal T_k(x-y)\cdot B^j[\nabla^ju,\nabla^{k-j}\eta_R](y)\;dy,
\end{align*}
where $B^j$, for $j=0, \dots , k-1$, are bilinear operators acting on vector spaces of appropriate dimension so that this formula preserves the product rule. 
Set
$$E_k(R)(x) = \int_{\mathbb{R}^n} \mathcal T_k(x-y)\cdot \mathcal A_k(D)u(y)(\eta_R(y)-1) \;dy
$$
and 
$$E_j(R)(x)= \int_{\mathbb{R}^n}  \mathcal T_k(x-y)\cdot B^j[\nabla^ju,\nabla^{k-j}\eta_R](y)\;dy \quad \text{for $j=0, \dots , k-1$.}$$
Equation \eqref{dec24-26} will follow if we show that  
\begin{align}\label{dec24-23}
\lim_{R\to \infty} E_j(R)(x)  = 0
\end{align}
for $j =0,\ldots, k$.  In the case $j=k$, the bound
\begin{align}\label{dec24-21}
|\mathcal T_k(x-y)| \lesssim \frac{1}{|x-y|^{n-k}} \quad \text{for  $y \neq x$}
\end{align}
implies
\begin{align}\label{dec24-20}
|E_k(R)(x)| \lesssim \int_{\mathbb{R}^n\setminus B_R(x)}  \frac{|\mathcal A_k(D)u(y)|}{|x-y|^{n-k}}|\eta_R(y)-1| \;dy,
\end{align}
  up to multiplicative constants depending on $\mathcal T_k$.
The integrand on the right-hand side of the latter inequality tends to zero pointwise as $R\to \infty$. On the other hand, the assumption $\mathcal A_k(D)u \in L^1 (\rn, \rM) +L^{n/k,1}(\rn, \rM)$ and H\"older's inequality in  Lorentz spaces 
ensure that
\begin{align*}
\chi_{\rn \setminus B_1(x)} \frac{|\mathcal A_k(D)u(y)|}{|x-y|^{n-k}} \in  L^1(\mathbb{R}^n).
\end{align*}
Thus, the right-hand side of the inequality \eqref{dec24-20} converges to $0$ as $R\to \infty$ by the dominated convergence theorem, and \eqref{dec24-23} follows for $j=k$.
\\
Turning our attention to the cases $j=0,\ldots, k-1$, thanks to equations  \eqref{dec24-22} and \eqref{dec24-21},
and the fact that $\nabla^{k-j} \eta$ is supported in the annulus $B_{3R/2}(x)\setminus B_R(x)$, we have
\begin{align*}
|E_j(R)(x)| \lesssim \frac{1}{R^{n-j}}\int_{B_{3R/2}(x)\setminus B_R(x)} |\nabla ^ju| \;dy \lesssim R^j \dashint_{B_{2R'} \setminus B_{R'}} |\nabla ^ju|\, dy
%
\end{align*}
for sufficiently large $R$, where  $R'=R-|x|$,  up to multiplicative constants depending on $T_k$ and $j$. Hence, equation \eqref{dec24-23} holds also for every $j=0, \dots , k-1$. This proves equation \eqref{dec24-26}. The proof of the identity \eqref{representation} for every $u\in V_\circ^{\mathcal A_k} (L^1 +L^{n/k,1})(\rn, \mathbb R^\ell)$ is complete.
\\ 
Now we use the cancellation condition on $\mathcal A_k(D)$ to find a co-canceling annihilator.  In particular, by \cite[Proposition 4.2 on p.~888]{VS3}, there exists a linear homogeneous co-canceling operator  $\mathcal L(D)$ of some order $k'$ such that $\mathcal L(D) \mathcal A_k(D)u=0$.
{
 As in \cite[Lemma 2.5]{VS3}, one has that
 $$
 \mathcal L(D)\mathcal A_k(D)u=\sum_{\beta \in \mathbb N^n, |\beta|=k'} L_\beta\partial^\beta \mathcal A_k(D)u=\sum_{\beta \in \mathbb N^n, |\beta|=k'} \partial^\beta (L_\beta \mathcal A_k(D)u)   =0
 $$ 
 for  suitable linear maps $L_\beta \in\operatorname{Lin}(\R^m, \R^l)\simeq\R^{l\times m}$ independent of $u$, and suitable  $l \in \mathbb N$.
 Let us write $L\mathcal A_k(D)u = [L_\beta \mathcal A_k(D)u]_{|\beta|=k'} \in \operatorname{Lin} (\mathbb{R}^n, \R^{N\times l})$ for the collection of $l$ maps with values in $\R^{N}$.  In particular, one can regard $[L_\beta \mathcal A_k(D)u]$ as $l$ rows  $\{(L_\beta \mathcal A_k(D)u)_i\}_{i=1}^l$ such that $\operatorname*{div}_{k'} (L_\beta \mathcal A_k(D)u)_i=0$ for each $i=1,\ldots,l$. 
\\
Owing to \cite[Lemma 2.5]{VS3} there exist a family of maps $K_\beta \in \operatorname{Lin}(\mathbb{R}^l, \mathbb{R}^m)$ such that
%
}
%
\begin{align}\label{expansion}
\mathcal A_k(D)u &= \sum_{|\beta|=k'} K_\beta L_\beta \mathcal A_k(D)u.
\end{align}
We can compactly combine the representation \eqref{representation} and the expansion \eqref{expansion} as
\begin{align}\label{another_expansion}
u  = \sum_{|\beta|=k'}  I_k T^\beta L_\beta \mathcal A_k(D)u,
\end{align}
where $T^\beta = T K_\beta$ is the Calder\'on-Zygmund operator defined as composition of the Calder\'on-Zygmund operator $T$ and the finite dimensional linear map $K_\beta$.
\\ Hence, an application of Theorem \ref{K-CZ} with the choice $T^\beta=TK_\beta$ and $F^\beta=[L_\beta \mathcal A_k(D)u]$
tells us that
\begin{align}\label{nov140}
\int_0^t  s^{-\frac{k}{n}} u^*(s)\, ds
\leq c\int_0^{t}  s^{-\frac{k}{n}} \int_{  s}^{\infty} [ L_\beta \mathcal A_k(D)u]^*(r)r^{-1+k/n}\,drds  \quad\text{for $t>0$,}
\end{align}
  for some constant $c=c(\mathcal A_k)$.
The inequality \eqref{sobolev1} follows from \eqref{nov140}, via Lemma \ref{KermanPick}.}
\end{proof}

\begin{proof}[Proof of Theorem \ref{sobolev-opt-canc}] By \cite[Theorem~4.4]{EMMP}, the inequality \eqref{sobolev3} holds with $Y(0,\infty)= X_k (0, \infty)$. Thus, the inequality \eqref{sobolev-bound-opt} follows from an application of Theorem \ref{characterization}.
\end{proof}

\begin{proof}[Proof of Corollary \ref{characterization-omega}] The implication \lq\lq (i) $\Rightarrow$ (ii)" trivially holds thanks to the inequality \eqref{dec24-30}.
\\ A proof of the implication \lq\lq (ii) $\Rightarrow$ (iii)" can be found in  \cite{KermanPick} and \cite{cianchi-pick-slavikova}.
\\
The implication  \lq\lq(iii) $\Rightarrow$ (i)" can be deduced  via an extension argument. Assume that the inequality \eqref{sobolev3om} holds.
We claim that 
\begin{equation}\label{sep27}
\bigg\|\int_s^\infty r^{-1+\frac k n}f(r)\, dr \bigg\|_{Y^e(0,\infty)} \leq c_1 \|f\|_{X^e(0, \infty)}
\end{equation}
for every $f \in \Mpl(0,\infty)$ with ${\rm supp} f \subset [0, |\Omega|]$, where $X^e(0, \infty)$ and $Y^e(0, \infty)$ denote the extended function norms defined as in \eqref{sep25}. Indeed, the inequality \eqref{sep27} can be verified via the following chain:
\begin{align}
    \label{sep28}
    \bigg\|\int_s^\infty r^{-1+\frac \alpha n}f(r)\, dr \bigg\|_{Y^e(0,\infty)} & = 
    \bigg\|\bigg(\chi_{[0,|\Omega|]}(\cdot)\int_{(\cdot)}^\infty r^{-1+\frac \alpha n}f(r)\, dr \bigg)^*(s)\bigg\|_{Y^e(0,\infty)}
    \\ \nonumber & = \bigg\| \int_{s}^{|\Omega|} r^{-1+\frac \alpha n}f(r)\, dr  \bigg\|_{Y(0,|\Omega|)}
    \\ \nonumber & \leq c_1 \|f\|_{X(0, |\Omega|)} = c_1\|f^*\|_{X(0, |\Omega|)} = c_1\|f\|_{X^e(0, \infty)}.
\end{align}
 An application of Theorem \ref{characterization} to the spaces $X^e(\rn, \mathbb R^m)$ and $Y^e(\rn, \R^\ell)$, and to the function $u^e: \rn \to \mathbb R^\ell$ obtained by extending $u$ by $0$ in $\rn \setminus \Omega$ tells us that  
\begin{equation}\label{sobolev1ext}
\|u^e\|_{Y^e(\rn, \R^\ell)} \leq c \|\mathcal A_k(D)u^e\|_{X^e(\rn, \rM)}
\end{equation}
for some constant   $c=c(c_3, \mathcal A_k)$.    By the definition of $X^e(\rn, \rM)$, $Y^e(\rn, \R^\ell)$, and $u^e: \rn \to \R^\ell$, one has that
$$\|\mathcal A_k(D)u^e\|_{X^e(\rn, \rM)}= \|(\mathcal A_k(D)u^e)^*\|_{X^e(0,\infty)}=
\|(\mathcal A_k(D)u^e)^*\|_{X(0,|\Omega|)}=\|(\mathcal A_k(D)u)^*\|_{X(\Omega, \rM)}$$
and, analogously
$$\|u^e\|_{Y^e(\rn, \R^\ell)}= \|u\|_{Y^e(\Omega, \R^\ell)}.$$
Therefore, the inequality \eqref{sobolev1om} follows from \eqref{sobolev1ext}.
\end{proof}

\begin{proof}[Proof of Corollary \ref{sobolev-bound-opt}] An application of \cite[Theorem~4.5]{EKP} tells us that the inequality \eqref{sobolev3om} holds with $Y(0,|\Omega|)= X_k (0, |\Omega|)$. The inequality \eqref{sobolev-bound-opt-om} is  thus a consequence of Theorem \ref{characterization}.
\end{proof}

\begin{proof}[Proof of Theorem \ref{sobolev-orlicz-canc}] 
By Theorems \ref{characterization} and \ref{characterization-omega}, 
the conclusion from the inequality:
\begin{equation}\label{orlicz5}
\bigg\|\int_s^{L} f (r)r^{-1+\frac k{n-k}} dr \bigg\|_{L^{A_{\frac nk}}(0, L)} \leq  c\|f\|_{L^A(0, L)}
\end{equation}
for some constant $c=c(A, \frac nk, L)$ and for every $f \in L^A(0, L)$, and the optimality of the Orlicz target space $L^{A_{\frac nk}}(0, L)$  in \eqref{orlicz5}. 
Here, $L$ equals either $\infty$ or $|\Omega|$, and $A$ and $A_{\frac nk}$ are as in Part (i) or (ii), respectively. Moreover, $c$  depends only on $\frac nk$ if $L=\infty$.
The inequality \eqref{orlicz5} follows from \cite[Inequality (2.7)]{cianchi_CPDE}. As shown in \cite[Lemma 2]{cianchi_PJM}, the latter inequality is equivalent to the one contained \cite[Lemma 1]{cianchi_IUMJ}, where the optimality of the Orlicz target space is also established.
\end{proof}

\begin{proof}
[Proof of Theorem \ref{sobolev-orliczlorentz-canc}]
Owing to Theorem \ref{characterization-omega}, the assertions of the statement follow from the inequality:
\begin{equation}\label{orliczri2}
\bigg\|\int_s^{|\Omega|} f (r)r^{-1+\frac k{n-k}} dr \bigg\|_{L(\widehat A, \frac nk)(0, |\Omega|)} \leq  c\|f\|_{L^A(0, |\Omega|)}
\end{equation}
for some constant $c=c(A, \frac nk, |\Omega|)$ and every   $f\in L^A(0, |\Omega|)$, and the optimality  of space 
$L(\widehat A, \frac nk)(0, |\Omega|)$  in \eqref{orliczri2} among all rearrangement-invariant spaces. 
These facts are in turn consequences of   \cite[Inequality (3.1) and Proof of Theorem 1.1]{cianchi_ibero}.
\end{proof}

\bigskip

\bigskip{}{}

 \par\noindent {\bf Data availability statement.} Data sharing not applicable to this article as no datasets were generated or analysed during the current study.

\section*{Compliance with Ethical Standards}\label{conflicts}

\smallskip
\par\noindent
{\bf Funding}. This research was partly funded by:
\\
(i) Grant BR 4302/3-1 (525608987) by the German Research Foundation (DFG) within the framework of the priority research program SPP 2410 (D. Breit);\\
(ii) Grant BR 4302/5-1 (543675748) by the German Research Foundation (DFG) (D. Breit);
\\ (iii) GNAMPA   of the Italian INdAM - National Institute of High Mathematics (grant number not available)  (A. Cianchi);
\\ (iv) Research Project   of the Italian Ministry of Education, University and
Research (MIUR) Prin 2017 ``Direct and inverse problems for partial differential equations: theoretical aspects and applications'',
grant number 201758MTR2 (A. Cianchi);
\\ (v) Research Project   of the Italian Ministry of Education, University and
Research (MIUR) Prin 2022 ``Partial differential equations and related geometric-functional inequalities'',
grant number 20229M52AS, cofunded by PNRR (A. Cianchi);
\\  (vi) National Science and Technology Council of Taiwan research grant numbers 110-2115-M-003-020-MY3/113-2115-M-003-017-MY3 (D. Spector);
\\ (vii) Taiwan Ministry of Education under the Yushan Fellow Program (D. Spector).

\bigskip
\par\noindent
{\bf Conflict of Interest}. The authors declare that they have no conflict of interest.

\end{document}